\documentclass[preprint,12pt]{elsarticle}
\usepackage{caption}
\usepackage{float}  
\usepackage{subcaption}
\usepackage{placeins}
\usepackage{afterpage}
\usepackage{amsmath}
\usepackage{graphicx}
\usepackage[hidelinks]{hyperref}  
\usepackage[capitalize]{cleveref}
\usepackage{mathrsfs}  
\usepackage{amsfonts}  
\usepackage{amssymb}   
\usepackage{enumitem}  
\usepackage{geometry}  
\usepackage{amsthm}  
\usepackage[numbers]{natbib}  
\newtheorem{theorem}{Theorem}[section]
\newtheorem{lemma}[theorem]{Lemma}

\newtheorem{definition}[theorem]{Definition}

\theoremstyle{definition}  
\newtheorem{example}{Example}[section]
\newtheorem{remark}{Remark}

\makeatletter
\def\ps@pprintTitle{%
   \let\@oddhead\@empty
   \let\@evenhead\@empty
   \let\@oddfoot\@empty
   \let\@evenfoot\@empty
}
\makeatother

\begin{document}

\allowdisplaybreaks
\begin{frontmatter}
	\title{
  \textbf{The Onsager-Machlup functional for distribution dependent SDEs driven by fractional Brownian motion}\tnoteref{1}
}

\author[inst1]{Yanbin Zhu}
\ead{zhuyb23@mails.jlu.edu.cn}
 \author[inst1]{Xiaomeng Jiang\corref{cor1}}
 \ead{jxmlucy@hotmail.com}
 \author[inst1,inst2]{Yong Li}
 \ead{liyong@jlu.edu.cn}
 \address[inst1]{College of Mathematics, Jilin University, Changchun 130012, PR China}
 \address[inst2]{School of Mathematics and statistics and Center for Mathematics and Interdisciplinary Sciences, Northeast Normal University, Changchun 130024, PR China}
 \cortext[cor1]{Corresponding author.}
 
 \tnotetext[1]{The second author (X. Jiang) is supported by National Key R\&D Program of China  (No. 2023YFA1009200). The third author (Y. Li) is supported by National Natural Science Foundation of China (Grant Nos. 12471183, 12071175).}
   \begin{abstract}
    In this paper, we compute the  Onsager-Machlup functional for distribution dependent SDEs driven by fractional Brownian motions with Hurst 
    parameter $H\in (\frac{1}{4},1)$. In the case \( \frac{1}{4} < H < \frac{1}{2} \), the norm can be either the supremum norm or Hölder norms of order \( \beta \) with \( 0 < \beta < H - \frac{1}{4} \). In the case \( \frac{1}{2} < H < 1 \), the norms can be a Hölder norm of order \( \beta \) with \( H - \frac{1}{2} < \beta < H - \frac{1}{4} \). As an example, we compute the Onsager-Machlup functional for the stochastic pendulum equation and derive its most probable path.
    \end{abstract}
 

\end{frontmatter}


	\section{Introduction}

   Consider the following distribution-dependent stochastic differential equation in $\mathbb{R}$:
\begin{equation}\label{fir}
    X_t = x + \int_0^t b(X_s, \mathcal{L}_{X_s}) \, ds + B^H_t,
\end{equation}
where $\mathcal{L}_{X_s}$ denotes the law of $X_s$, and $B^H = \{B^H_t, t \in [0,1]\}$ is a fractional Brownian motion with Hurst parameter $H \in (0,1)$. When $H = \frac{1}{2}$, the process reduces to a standard Brownian motion. For $H > \frac{1}{2}$, the fractional Brownian motion belongs to the regular regime, characterized by smoother trajectories and positive correlation. In contrast, when $H < \frac{1}{2}$, it falls into the singular regime, exhibiting rougher paths with negative correlation and increased irregularity.

	The purpose of this paper is to investigate the limiting behavior of the following ratio as
	$\varepsilon\to 0$:
	\begin{equation}
		\gamma_\varepsilon(\phi) = \frac{P(\|X - \phi\| \leq \varepsilon)}{P(\|B^H\| \leq \varepsilon)}, \label{ratial}
	\end{equation}
	where $\phi$ is a function such that $\phi-x$ belongs to the Cameron-Martin space associtated with the fractional Brownian motion and $\Vert \cdot \Vert$ is a suitable norm. When the limit exists and can be expressed as $\exp(J(\phi))$, the functional $J$ is referred to as the Onsager-Machlup functional associated with  equation \eqref{fir} and the norm $\Vert \cdot \Vert$.
	
		The Onsager-Machlup functional describes the most probable paths of diffusion processes and is analogous to the Lagrangian of a dynamical system, which is used to characterize the optimal evolution trajectory of a particle along a given path. A fundamental challenge is to understand how systems transition between  states under stochastic perturbations. Given two states \(x_1\) and \(x_2\), and \(\phi\) taken from an appropriate function space that satisfies the path from the initial state \(x_1\) to the final state \(x_2\), the path that maximizes \( \exp(J(\phi)) \) can be obtained using the variational method. This path is then the most probable path for the transition from state \(x_1\) to \(x_2\). Therefore, the Onsager-Machlup functional is a deterministic quantity that characterizes the most probable transition path between the  states of the system using the variational principle. It has significant applications in physics and chemistry (see \cite{BATTEZZATI2013163,Durr1979}).

		The Onsager-Machlup functional was first introduced by Onsager and Machlup \cite{Onsager1953} in 1953 and was subsequently solved by Ikeda and Watanabe \cite{ikeda1981stochastic} for $\phi \in C^2([0,1], \mathbb{R}^d)$ under the supremum norm. Shepp and Zeitouni \cite{Shepp1992} extended the result to the case where $\phi - x$ belongs to the Cameron-Martin space.

In \cite{Capitaine1995}, Capitaine further extended this result to a broader class of natural norms on the Wiener space, including Hölder, Sobolev, and Besov norms. Moret and Nualart \cite{Nualart} studied the Onsager-Machlup functional for stochastic differential equations  driven by fractional Brownian motion. Maayan \cite{maayan2017onsagermachlupfunctionalassociatedadditive} further extended the fractional results to finite-dimensions.

Recent developments have demonstrated the expanding scope of this concept. Liu et al. \cite{LIU2023107203} derived the Onsager-Machlup functional for McKean-Vlasov SDEs in a class of norms that dominate $L^2([0,1], \mathbb{R}^d)$. Huang et al. \cite{levy} established a closed-form expression for the Onsager–Machlup functional for jump-diffusion processes with finite jump activity. Zhang and Li \cite{zhang2024onsagermachlupfunctionalstochasticdifferential} investigated the Onsager-Machlup functional for stochastic differential equations with time-varying noise in the Hölder norm, where $0 < \alpha < \frac{1}{4}$. Carfagnini and Wang \cite{Carfagnini2024} interpret Loewner energy as the Onsager-Machlup functional for {$\text{SLE}_{\kappa}$} loop measure for any fixed $\kappa\in (0,\frac{1}{4}]$.
	Li and Li \cite{doi:10.1137/20M1310539} proved that the $\Gamma$-limit of the OM functional on the space of curves is the geometric form of the FW functional in a proper time scale $T=T(\varepsilon)$ as $\varepsilon \to 0$. Liu \cite{liu2023onsagermachlupactionfunctionaldegenerate} studied the Onsager-Machlup action functional for degenerate stochastic differential equations driven by fractional Brownian motion.

	The distribution dependent equation, also known as the McKean-Vlasov equation or mean-field SDE, was first introduced by Henry P. McKean in \cite{McKean1966}. The McKean-Vlasov equation is fundamental for modeling mean-field interactions in large systems, with applications spanning physics, finance, control theory and machine learning. 	 Galeati et al. \cite{Galeati2021} studied distribution dependent stochastic differential equations with irregular, possibly distributional drift, driven by an additive fractional Brownian motion. 
	Fan et al. \cite{Fan2022} proved the well-posedness of distribution dependent stochastic differential equations driven by fractional Brownian motions, and then establish a general result on the Bismut formula for the Lions derivative by using Malliavin calculus. 
	
	This paper focuses on deriving the Onsager-Machlup functional for distribution dependent SDEs driven by fractional Brownian motions. The fractional-order distribution dependent SDE is a more general form of the standard SDE. Its advantage lies in its ability to capture memory effects and long-term dependencies within the system, as well as describe distribution-dependent randomness, thus enabling more accurate modeling of dynamic behaviors that depend on the system’s historical states in complex systems. 
	
			Compared to the standard Onsager-Machlup functional, the equation in this paper introduces a distribution dependent drift term and extends the noise from Brownian motion to the more general fractional Brownian motion. 	The first challenge lies in the Girsanov transformation. The distribution of a stochastic process depends on the probability measure on the probability space. Therefore, in the Girsanov transformation, we freeze the distributional coefficient of the drift term. Moreover, since the noise is fractional Brownian motion, the operator \( (K^H)^{-1} \), defined by the covariance kernel of fractional Brownian motion, involves fractional derivatives or fractional integrals due to this transformation, which requires significantly more intricate estimates.
  The second challenge lies in computing the conditional expectation of the term  
\[
\exp\left[\int_0^1 (K^H)^{-1}\left(\int_0^\cdot b\left(B^H_r + \phi_r, \mathcal{L}_{X_r}\right)dr\right)(s)dW_s\right]
\]  
under the condition $\Vert B^H \Vert \leq \varepsilon$. Although $\mathcal{L}_{X_s}$ depends on $B^H$, the law $\mathcal{L}_{X_s}$ remains unchanged under the small-noise constraint $\Vert B^H \Vert \leq \varepsilon$. To resolve this, we extend the approach in \cite{Nualart} by performing a Taylor expansion on $b\left(B^H_s + \phi_s, \mathcal{L}_{X_s}\right)$ solely with respect to the spatial variable.  Notably, we avoid using the Itô formula framework in \cite{LIU2023107203} due to the lack of a well-established extension of Malliavin calculus to distribution-dependent scenarios. Instead, by combining the fractional Girsanov transformation and Taylor expansion, we circumvent the need for Malliavin calculus altogether. Furthermore, by employing variational principles, we derive the Euler-Lagrange equation of the Onsager-Machlup functional.

	The paper is structured as follows. In Section 2, we establish essential notations, formulate underlying assumptions, and review key preliminary results. In Section 3, we introduce a method for simplifying the Onsager-Machlup functional via the Girsanov transformation and present our main theorem. We subsequently provide rigorous proofs for both regular and singular cases. In Section 4, we extend the Onsager-Machlup functional to the finite-dimensional case. Finally, in Section 5, we derive the Euler-Lagrange equation for the Onsager-Machlup functional and verify our results through two specific examples.

\section{Preliminaries}
	In this section, we recall some basic notations, assumptions, lemmas that will be used later.

   \subsection{Fractional calculus}
   
   We recall some definitions and results about fractional calculus .
   
   Given $p\in [1,\infty]$ and $[a,b] \subset \mathbb{R},$ we denote by 
   $L^p([a,b])=L^p([a,b],\mathbb{R})$ the space of measurable functions $f:[a,b]\to \mathbb{R}$ such that 
   \begin{equation*}
   	\Vert f\Vert _{L^p([a,b])}=\left(\int_a^b |f(t)|^p dt\right)^\frac{1}{p} <\infty
   \end{equation*} 
   with usual modification for $p=\infty.$ For simplicity, we will use $\Vert \cdot \Vert_\infty$ to denote $\Vert \cdot \Vert_{L^\infty}$.

   The space $C^\alpha_0([0,T])$ represents the set of functions in $C^\alpha([0,T])$ that satisfy 
	$f(0)=0$. Let $\Vert \cdot \Vert_\alpha$ be the 
	Hölder norm of order $\alpha$ on $C^\alpha_0([0,T])$ defined by 
	\begin{equation*}
		\Vert f \Vert_\alpha =\sup_{0\leq s<t\leq T}\frac{|f(t)-f(s)|}{|t-s|^\alpha}.
	\end{equation*}

   The following definitions of fractional integrals and fractional derivatives are from \cite[Definition 2.1]{Samko1993} and \cite[Definition 2.2]{Samko1993}.   
   
   \begin{definition}
   	Let $f\in L^1([a,b])$. The integrals 
   	\begin{align*}
   		(I_{a^+}^\alpha f)(x)&:=\frac{1}{\Gamma(\alpha)}\int_a^x (x-y)^{\alpha-1}f(y)dy,\quad x\geq a,\\
   		(I_{b^-}^\alpha f)(x)&:=\frac{1}{\Gamma(\alpha)}\int_x^b (y-x)^{\alpha-1}f(y)dy,\quad x\leq b,
   	\end{align*}
   	where $\alpha>0$, are respectively called right and left fractional integral of the order $\alpha$.
   \end{definition}
   
   For any $\alpha\geq 0,$ any $f\in L^p([a,b])$ and $g\in L^q([a,b])$ where $1/p+1/q\leq \alpha,$ we have:
   \begin{equation}
   	\int_a^b f(s)(I_{a^+}^\alpha g)(s)ds=\int _a^b (I_{b^-}^\alpha f)(s) g(s) ds.\label{ffubini}
   \end{equation}
   
   	
   	If $1 \leq p < \infty$, we denote by $I_{a^+}^\alpha(L^p)$ the image of $L^p([a,b])$ under the operator $I_{a^+}^\alpha$. Similarly, $I_{b^-}^\alpha(L^p)$ can be defined.

   	\begin{definition}
   		 Let $f\in I_{a^+}^\alpha(L^p) $, $g\in I_{b^-}^\alpha(L^p) $.
   		 Each of the expressions 
   		 \begin{align*}
   		(D_{a^+}^\alpha f)(x)&:=\left(\frac{d}{dx} \right)^{[\alpha]+1}I_{a^+}^{1+[\alpha]-\alpha}f(x),\\
   		(D_{b^-}^\alpha g)(x)&:=\left(-\frac{d}{dx} \right)^{[\alpha]+1}I_{b^-}^{1+[\alpha]-\alpha}g(x),
   	   	\end{align*}
   		 are respectively called the right and left fractional derivative.
   \end{definition}
  From \eqref{ffubini}, we deduce the formula  
\begin{equation}
  \int_a^b f(s)(D_{a^+}^\alpha g)(s)ds
  =\int _0^t (D_{b^-}^\alpha f)(s) g(s) ds,\quad 0<\alpha<1, \label{fffubini}
\end{equation}
which holds under the assumptions that \( f\in I^\alpha_{b^-}(L^p) \) and \( g\in I^\alpha_{a^+}(L^q) \) satisfy \( 1/p+1/q\leq 1+\alpha \).
   
	If $f\in I_{a^+}^\alpha(L^p) $, the function $ \phi $ such that $f=I_{a^+}^\alpha(\phi)$ is unique in $L^p $. Fractional derivatives can be regarded as the inverse operation of fractional integrals.

  When $\alpha p > 1$, functions in $I_{a^+}^\alpha(L^p)$ are $\alpha - \frac{1}{p}$-Hölder continuous. Any Hölder continuous function of order $\beta > \alpha$ has a fractional derivative of order $\alpha$, see \cite[Proposition 2.1]{Decreusefond1999}.

	Although fractional derivatives are defined as the derivatives of fractional integrals, they can also be directly expressed using the Weyl representation (see \cite[Remark 5.3]{Samko1993}): 
\begin{equation}
D^\alpha_{0^+}f(x) = \frac{1}{\Gamma(1-\alpha)} \left( \frac{f(x)}{x^\alpha} + \alpha \int_0^x \frac{f(x) - f(y)}{(x - y)^{\alpha+1}}  dy \right), \quad 0<\alpha<1,
\label{Weil}
\end{equation}
where the improper integral converges in the $L^p$ sense. Therefore, if \( f \) has Hölder continuity with an exponent strictly greater than \( \alpha \), then the fractional derivative of order \( \alpha \) exists.

  Similar to the integral operator, the fractional integral operator is also a bounded linear operator on $L^p$ spaces, see \cite[Theorem 2.6]{Samko1993} .
\begin{theorem} For any $f \in L^p([a,b])$, $1 \leq p < \infty$, one has
	 \begin{equation} \Vert I_{a^+}^\alpha(f)\Vert_{L^p} \leq \frac{(b-a)^\alpha}{|\Gamma(\alpha)|} \Vert f \Vert_{L^p}. \label{Ia} \end{equation} \end{theorem}

  \subsection{Fractional Brownian motion }
  	In this section, we will introduce the basic concepts and properties of fractional Brownian motion. 
  	
  	A real valued continuous process $\{B^H_t,\  t\in [0,T]\}$ is a fractional Brownian motion with Hurst parameter $H\in (0,1)$ if it is a centered Gaussian process with covariance function 
  	\begin{equation*}
  		\mathbb{E}[B^H_tB^H_s]=\frac{1}{2}(|t|^{2H}+|s|^{2H}-|t-s|^{2H}).
  	\end{equation*}
  	
  	If $H=\frac{1}{2}$, fractional Brownian motion corresponds to classical Brownian motion. But if $H\neq \frac{1}{2},$ it is not a semimartingale nor a Markov process;  its trajectories are $P-$a.s. in $C^{H-\varepsilon}_0([0,T])$ for any $\varepsilon>0$.
  	
  Given a fractional Brownian motion \( B^H \) with Hurst parameter \( H \) on a probability space \( (\Omega, \mathcal{F}, P) \), it is always possible to construct a standard Brownian motion \( W \) such that the following canonical representation holds. As shown in equation (42) in the work of Alòs et al. \cite{Alos2001}, the fractional Brownian motion can be expressed through the integral representation:
\begin{equation}
	B^H_t = \int_0^t K^H(t,s) \, dW_s,\label{kh int}
\end{equation}
where \( K^H(t,s) \) is a square-integrable kernel given by:
\begin{equation}
K^H(t,s) = c_H (t-s)^{H-\frac{1}{2}} + c_H \left(\frac{1}{2} - H\right) \int_s^t (\theta-s)^{H-\frac{3}{2}} \left(1-\left(\frac{s}{\theta}\right)^{\frac{1}{2}-H}\right) d\theta, 
\label{kh}
\end{equation}
with 
\[
c_H = \left(
\frac{2H\Gamma\left(
\frac{3}{2}-H
\right)}
{\Gamma\left(\frac{1}{2}+H\right)\Gamma(2-2H)}
\right)^{\frac{1}{2}}.
\]   
  When $H\geq \frac{1}{2},$ $K^H$ can be express as 
  \begin{equation}
  	K^H(s,u)=c_H\alpha u^{-\alpha}\int_u^s (\theta-u)^{\alpha-1}\theta^\alpha d\theta.\label{kh1/2}
  \end{equation}
  
 For the purpose of the Girsanov transformation, we define the operator \( K^H \) corresponding to \( K^H(t,s) \). The operator \( K^H \), mapping from \( L^2([0,1]) \) to \( I_{0^+}^{H+\frac{1}{2}}(L^2([0,1])) \), is given by:  
\begin{equation}
	(K^Hh)(t) = \int_0^t K^H(t,s) h(s) \, ds.\label{kh operator}
\end{equation}

 From \cite[Lemma 10]{Nualart}, the operator \( K^H \) can be expressed using fractional integrals as follows:   
\[
(K^H h)(s) = \begin{cases}
\displaystyle
I^{1-2\alpha}_{0^+} \, s^{\alpha} I^{\alpha}_{0^+} \, s^{-\alpha} h, & H \leq 1/2, \\
\displaystyle
I^{1}_{0^+} \, s^{\alpha} I^{\alpha}_{0^+} \, s^{-\alpha} h, & H \geq 1/2,
\end{cases}
\]  
where \(\alpha = |H - \frac{1}{2}|\) and \(h \in L^2([0,1])\). The inverse operator \((K^H)^{-1}\) is then defined as:  
\begin{align}  
\notag (K^H)^{-1} h &= s^{\alpha} D^{\alpha}_{0^+} \, s^{-\alpha} D^{1-2\alpha}_{0^+} h, \quad H \leq 1/2, \\  
(K^H)^{-1} h &= s^{\alpha} D^{\alpha}_{0^+} \, s^{-\alpha} h', \quad H \geq 1/2,\label{g1/2}   
\end{align}    
for all \(h \in I^{H+\frac{1}{2}}_{0^+}(L^2)\). If \(h\) is differentiable, the operator simplifies to:  
\begin{equation} \label{l1/2}  
(K^H)^{-1} h = s^{-\alpha} I^{\alpha}_{0^+} \, s^{\alpha} h', \quad H \leq 1/2.  
\end{equation}  
Moreover, the operator \((K^H)^{-1}\) preserves the adaptability property. \subsection{Approximate limits in Wiener space}
   Let $W=\{W_t,t\in [0,1] \}$ be a Wiener process defined in the canonical probability space $(\Omega,\mathcal{F},P)$, where 
   \begin{equation*}
   	\Omega=\{f\in C([0,1],\mathbb{R} ), f(0)=0   \},
   \end{equation*}
   and $P$ is the Wiener measure.
   
   Let $CM$ be the Cameron-Martin space:
   \begin{equation*}
   	CM=\{h\in \Omega,\text{ $h$ is absolutely continuous and $h'$}\in L^2([0,1])\}.
   \end{equation*}
   
  Using the fact that the supremum norm and Hölder norm are measurable norms on the Cameron-Martin space (see \cite[Lemma 6]{Nualart}), we derive the following two important theorems for the computation of the Onsager-Machlup functional.
   
     The first theorem comes from \cite[Theorem 2]{Nualart}.
   
  \begin{theorem}
   \label{thm2}
   If $\Vert \cdot \Vert$ is the supremum norm or Hölder norm of order $\beta$ with 
   $0 < \beta < H$, then
   \begin{equation}
      \lim_{\varepsilon \to 0} \mathbb{E} \left( \exp \left( \int_0^1 h(s) \, dW_s \right) \middle| \Vert B^H \Vert < \varepsilon \right) = 1, \quad \forall h \in H.
   \end{equation}
\end{theorem}

   Before presenting the next theorem, we first introduce the definition and properties of the trace.
   \begin{definition}[\cite{Nualart}]
   	  An operator $K : L^2([0,1]) \to L^2([0,1])$ is nuclear if and only if 
\[
\sum_{n=1}^\infty |\langle K e_n, g_n \rangle| < \infty,
\]
for all $B = (e_n)_n$, $\widetilde{B} = (g_n)_n$ orthonormal sequences in $L^2([0,1])$, where $\langle\cdot,\cdot\rangle$ denotes the inner product in $L^2$. We define the trace of
a nuclear operator $K$ by
\[
\operatorname{Tr} K = \sum_{n=1}^\infty \langle K e_n, e_n \rangle,
\]
for any $B = (e_n)_n$ orthonormal sequence in $L^2([0,1])$.
   \end{definition}
   
   The definition is independent of
the sequence we have chosen. Given a symmetric function $f \in L^2([0, 1]^2)$, the
Hilbert-Schmidt operator $K(f) : L^2([0,1])\to L^2([0,1])$ associated with $f$, defined by
\begin{equation} \label{eq:operator_definition}
(K(f))(h)(t) = \int_0^t f(t, u) h(u) \, du,
\end{equation}
is nuclear if and only if $\sum_{n=1}^\infty |\langle K e_n, e_n \rangle| < \infty$ for all $B = (e_n)_n$ orthonormal sequences in
$L^2([0,1])$. If $f$ is continuous and the operator $K(f)$ is nuclear, the trace of $f$ is defined as the trace of $K(f)$, given by:  
\[
\operatorname{Tr} f = \operatorname{Tr} K(f) = \int_0^1 f(s, s) \, ds.
\]   
   
   The second theorem comes from 
   \cite[Theorem 8]{Nualart}.
  \begin{theorem}
   Let $\Vert \cdot \Vert$ be the supremum norm or the Hölder norm of order $\beta$, where $0 < \beta < H$. Let $f$ be a symmetric function in $L^2([0,1]^2)$. If $K(f)$ is nuclear, then
   \begin{equation*}
      \lim_{\varepsilon \to 0} \mathbb{E}\left(\exp\left(2 \int_0^1 \int_0^t f(s,t) \, dW_s dW_t\right) \middle| \Vert B^H \Vert < \varepsilon\right) = \exp\left(-\operatorname{Tr}(f)\right). \label{thm3}
   \end{equation*}
\end{theorem}   

  The following theorems describe the small ball behavior of fractional Brownian motion under the supremum and Hölder norms. Their proofs can be found in \cite[Theorem 4, Theorem 5]{Nualart}.
     \begin{theorem}
   	Let $B^H$ be a fractional Brownian motion. Then there exists a positive constant $C_H$ such that
   	\begin{equation}\label{sup}
   		\lim_{\varepsilon\to 0} \varepsilon^\frac{1}{H}\ln P(\sup_{0\leq t\leq 1}|B^H|\leq \varepsilon )=- C_H.
   	\end{equation}
   \end{theorem}

   \begin{theorem}
   	Let $B^H$ be a fractional Brownian motion and $0\leq \beta <H.$ Then there exist constants $0<K_1\leq K_2<\infty$ depending on $H$ and $\beta$ such that for all $0<\varepsilon<1$
   	\begin{equation}
   		-K_2\varepsilon^{-\frac{1}{H-\beta}}\leq 
   		\ln P\left(\sup_{0\leq s<t\leq 1}\frac{|B^H_t-B^H_s|}{|t-s|^\beta} \leq \varepsilon\right)\leq -K_1\varepsilon^{-\frac{1}{H-\beta}}.\label{sbb holder}
   	\end{equation}
   \end{theorem}
   
   The following theorem allows us to decompose the Onsage-Machlup functional into multiple terms during the calculation (see \cite[Lemma 8]{Nualart}).
   \begin{theorem}
   	For a fixed $n\geq 1$, let $z_1,\dots,z_n$ be random variables defined on $(\Omega, \mathcal{F}, P)$ and $\{A_\varepsilon; \varepsilon>0\}$ a family of sets in $\mathcal{F}$. Suppose that for any $c\in \mathbb{R}$ and any $i=1,\dots,n,$ we have 
   	\begin{equation*}
   		\limsup_{\varepsilon\to 0} \mathbb{E}(\exp(cz_i)\mid A_\varepsilon)\leq1.
   	\end{equation*}
   	Then 
   	\begin{equation*}
   		\lim_{\varepsilon\to 0}\mathbb{E} \left(\exp\left(\sum_{i=1}^nz_i\right)\mid A_\varepsilon\right)=1.
   	\end{equation*}\label{lem8}
   \end{theorem}

     \subsection{Distribution dependent SDEs}
   For $\theta\in [1,\infty),$ let $\mathscr{P}_\theta(\mathbb{R})$ be the space of probability measures on $\mathbb{R}$ with finite $\theta$-th moment. We define the $L^\theta$-Wasserstein distance on $\mathscr{P}_\theta(\mathbb{R})$ by
   \begin{equation}
   	d_\theta(\mu,\nu)=\inf_{\pi\in \Pi(\mu,\nu)} \left(\int_{\mathbb{R}\times\mathbb{R}} |x-y|^\theta \pi(dx,dy) \right)^\frac{1}{\theta},\quad \mu,\nu\in \mathscr{P}_\theta(\mathbb{R}),
   \end{equation}
   where $\Pi(\mu,\nu)$ stands for the set of all possible couplings of $(\mu,\nu).$    
   For a random variable $X$, we use $\mathcal{L}_{X}$ to denote the distribution of $X$ on $\mathbb{R}$.
 \begin{remark}
    \begin{itemize}    
        \item[(1)] For any random variables \( X, Y \in L^\theta(\Omega) \), the following inequality holds:
        \[
        d_\theta(L_X, L_Y) \leq \left( \mathbb{E} \mid X - Y \mid^\theta \right)^{\frac{1}{\theta}}.
        \]
        
        \item[(2)] If \( \phi_t \) is a deterministic path, then the law of \( \phi_t \) is a Dirac measure at \( \phi_t \), i.e., 
        \[
        L_{\phi_t} = \delta_{\phi_t}.
        \]
    \end{itemize}
\end{remark}
			
	In this paper, we consider the distribution dependent SDE \eqref{fir}, where $b:\mathbb{R}\times  \mathscr{P}_2(\mathbb{R})\to  \mathbb{R}.$

	A function $f(x, \mu): \mathbb{R} \times \mathscr{P}_2(\mathbb{R}) \to \mathbb{R}$ is said to be Lipschitz continuous with constant $L$ if it satisfies  
\begin{equation*}  
    |f(x, \mu) - f(y, \nu)| \leq L \big( |x - y| + d_2(\mu, \nu) \big),  
    \quad \forall x, y \in \mathbb{R}, \, \mu, \nu \in \mathscr{P}_2(\mathbb{R}).  
\end{equation*}  
	
	Next, we impose some assumptions on $b$.
	\begin{itemize}
    \item (\textbf{Hyp 1}) 
    \begin{enumerate}
    	\item The derivatives of \( b \) with respect to the spatial variable  \( b_x \)  exists and continuous;
    	\item  \( b \) and \( b_x \) are  bounded;
		 \item $b$ is Lipschitz continuous with constant $L_1$;
		 \item There exists a constant $k$ such  
		that 
		\begin{equation*}
			|b(x+h,\mu)-b(x,\mu)-b_x(x,\mu)h|\leq k h^2,\quad \forall x,h\in \mathbb{R}, \mu \in \mathscr{P}_2(\mathbb{R}).
		\end{equation*}
		
                   \end{enumerate}
\end{itemize}
 
 \begin{itemize}
    \item (\textbf{Hyp 2}) 
    \begin{enumerate}
    	\item The derivatives of \( b \) with respect to the spatial variable  \( b_x \) and \( b_{xx} \) exist;
    	\item  \( b \), \( b_x \), and \( b_{xx} \) are all bounded;
		\item $b,b_x$ and $b_{xx}$ are Lipschitz continuous with constants 
        $L_1$, $L_2$ and $L_3$.
        \end{enumerate}
        \end{itemize}

\begin{remark}
	From \cite{Fan2022}, it is known that under the conditions of assumptions \textnormal{(\textbf{Hyp 1})} and \textnormal{(\textbf{Hyp 2})}, the existence and uniqueness of the strong solution to equation \eqref{fir} hold, and that $\mathcal{L}_{X_t} \in \mathscr{P}_2(\mathbb{R})$, for $t \in [0,1].$
	\end{remark}

\subsection{Technical theorems}
	In this section, we will introduce several commonly utilized technical lemmas and theorems.
	
   We first introduce the fractional Girsanov transformation. Let \( B^H \) be a fractional Brownian motion with Hurst parameter \( 0 < H < 1 \) and denote by \( \{\mathcal{F}^{B^H}_t,\ t \in [0,T] \} \) its natural filtration. Given an adapted process with integrable trajectories \( u = \{u_t,\ t \in [0,T] \} \), consider the transformation 
\begin{equation}
    \tilde{B}^H_t = B^H_t + \int^t_0 u_s \, ds. \label{gt}
\end{equation}
We can write 
\begin{align}
    \tilde{B}^H_t 
    &= B^H_t + \int^t_0 u_s \, ds \notag \\
    &= \int_0^1 K^H(t,s) \, dW_s + \int_0^t u_s \, ds \notag \\
    &= \int_0^1 K^H(t,s) \, d\tilde{W}_s,\end{align}
where
\begin{equation}
    \tilde{W}_t = W_t + \int_0^t (K^H)^{-1} \left( \int_0^\cdot u_s \, ds \right)(r) \, dr.
\end{equation}
The following theorem is taken from \cite[Theorem 2]{NUALART2002103}.

\begin{theorem}[Girsanov transform for fractional Brownian motion]\label{Girsanov}
    Consider the shifted process \eqref{gt} defined by a process \( u = \{u_t,\ t \in [0,T] \} \) with integrable trajectories. Assume that:
    \begin{enumerate}[label=\roman*)]
        \item \( \int_0^\cdot u_s \, ds \in I_{0^+}^{H+\frac{1}{2}}(L^2([0,T])) \), almost surely;
        \item \( \mathbb{E}(\xi_T) = 1 \), where 
        \begin{equation*}
            \xi_T = \exp\left( -\int_0^T (K^H)^{-1} \left( \int_0^\cdot u_s \, ds \right)(r) \, dW_r - \frac{1}{2} \int_0^T \left[ (K^H)^{-1} \left( \int_0^\cdot u_s \, ds \right)(r) \right]^2 \, dr \right). 
        \end{equation*}
    \end{enumerate}
    Then the shifted process \( \tilde{B}^H \) is an \( \mathcal{F}_t^{B^H} \)-fractional Brownian motion with Hurst parameter \( H \) under the new probability \( \tilde{P} \) defined by \( \frac{d\tilde{P}}{dP} = \xi_T \).
\end{theorem}

   The following theorem introduces the conditions under which a stochastic integral and a Lebesgue integral can be interchanged (see \cite[Theorem 2.2]{Veraar01082012}).
   \begin{theorem}\label{fubini}
   Let \((X, \Sigma, \mu)\) be a \(\sigma\)-finite measure space. Let \(\varphi: X \times [0,T] \times \Omega \to \mathbb{R}\) be progressively measurable, and suppose that for almost every \(\omega \in \Omega\),
\[
\int_X \left( \int_0^T |\varphi(x,t,\omega)|^2 \, dt \right)^{1/2} \, d\mu(x) < \infty.
\]
   	Then for almost all $\omega\in \Omega$ and for all $t\in [0,T]$, we have
   	\begin{equation*}
   		\int_X \int_0^T \varphi(x,t,\omega) dW_t d\mu(x)=\int_0^T \int_X \varphi(x,t,\omega) d\mu(x) dW_t .
   	\end{equation*}
   	  \end{theorem}

   The following exponential inequality for martingales is derived from \cite[Lemma 2.1]{inequality}.
   \begin{theorem}	Let $(\Omega,\mathcal{F},P)$ be a probability triple and let $\{M_t\}_{t\geq 0}$ be a locally square integrable martingale w.r.t. the filtration 
  $ \{\mathcal{F}_t\}_{t\geq 0}$, $M_0=0.$ 
  Let \( \langle M\rangle _t \) denote the quadratic variation of \( \{M_t\} \).  Suppose that $|\Delta M_t|=|M_t-M_{t^-} |\leq K$ for all $t>0$, $0\leq K<\infty.$
  Then for each $a>0$, $b>0$,
  \begin{equation*}
  	P(M_t\geq a \ \text{and}\  \langle M\rangle _t \leq b^2 \ \text{for some}\  t)
  	\leq \exp\left(-\frac{a^2}{2(aK+b^2)}\right).
  \end{equation*}\label{exponential ineq}
   \end{theorem}

\section{Main Results}

	In this section, we will compute the results for both the regular and singular cases. The main difference between the regular case and the singular case is that the operator $(K^H)^{-1}$
is expressed as a fractional-order derivative instead of a fractional-order integral, the fractional derivative operator is unbounded.
	Furthermore, this difference leads us to rely on Hölder norms during the proof process in the regular case, and the supreme norm will no longer hold in this case.
	Therefore, we will use the  representation \eqref{Weil} to estimate each term. To do this, we need to simplify the ratio \eqref{ratial}  by using the  Girsanov transform for fractional Brownian motion.

Consider the distribution dependent stochastic differential equation  \eqref{fir}.  
 Define $\mathcal{H}^p=\{K^Hh,h\in L^p([0,1])\} $, where $K^H$ is the operator from $L^p([0,1])$ to $I_{0^+}^{H+\frac{1}{2}}(L^p([0,1]))$, as defined by \eqref{kh operator}. The operator $K^H$ is an isomorphism from $L^2([0,1])$ to $I_{0^+}^{H+\frac{1}{2}}(L^2([0,1]))$ (see \cite[Theorem 2.1]{Decreusefond1999}). So the space $\mathcal{H}^2$ is included in the space of  
 	Hölder continuous functions of order $H$. 
 	 
 	We denote by $\dot{\phi}$ the function in $L^p([0,1])$ such that
 	\begin{equation}
 		K^H \dot{\phi} = \phi - x.  \label{kh exist}
 	\end{equation}
 	The following theorem shows that $k^H$ is invertible under  the conditions of the regular and singular cases.
 	
 	\begin{lemma}
 	\textbf{(Regular case)}  
When \( H > \frac{1}{2} \), for all \( \phi - x \in \mathcal{H}^2 \), there exists a unique \( \dot{\phi} \in L^2([0,1]) \). 

\textbf{(Singular case)}  
When \( \frac{1}{4} < H < \frac{1}{2} \) and \( p > \frac{1}{H} \), for all \( \phi - x \in \mathcal{H}^p \), there exists a unique \( \dot{\phi} \in L^p([0,1]) \) .	\end{lemma}
 \begin{proof}
 	The existence of $ \dot{\phi} $ is evident. Next, we prove its uniqueness.  

For the singular case, leveraging the isomorphism between 
\( L^2([0,1]) \) and \( I_{0^+}^{H+\frac{1}{2}}(L^2) \),
it is straightforward to conclude that $ \dot{\phi} $ is unique in the $ L^p([0,1]) $ sense.  

For the regular case, suppose that there exist $ \dot{\phi}_i $ such that $ K^H \dot{\phi}_i = \phi - x $ for $ i = 1, 2 $. Since $ L^p([0,1]) \subset L^2([0,1]) $, by the uniqueness established for $ p=2 $, it follows that $\dot{\phi_1} = \dot{\phi_2} $ almost surely. Consequently, $ \dot{\phi} $ is unique in the $ L^p([0,1]) $ sense.
	\end{proof}

 	 We aim to analyze the limiting behavior of the ratio  
\[
\gamma_\varepsilon(\phi) = \frac{P(\|X - \phi\| \leq \varepsilon)}{P(\|B^H\| \leq \varepsilon)},  
\]  
as \( \varepsilon \to 0 \), where \( \| \cdot \| \) denotes an appropriate norm.

 	To distinguish, we use $\mathcal{L}^P_{X_s}$ to denote the distribution of $X_s$ under probability $P$.
 Using \eqref{kh exist},  equation \eqref{fir} can be written as 
 	\begin{align}
 	\notag	X_t-\phi_t &=x-\phi_t +\int_0^t b(X_s,\mathcal{L}^P_{X_s} )ds+B^H_t \\ \notag
 		&=\int_0^t b(X_s,\mathcal{L}^P_{X_s} )ds+B^H_t-\int_0^t K^H(t,s)\dot{\phi_s}ds \\
 	\notag	&=B^H_t+\int_0^t b(X_s,\mathcal{L}^P_{X_s} )-K^H(t,s)\dot{\phi_s}ds.
 	\end{align}
 
 Next, we perform the Girsanov transformation with a fixed drift term and distribution-dependent coefficient. Set $Y_t=B^H_t+\phi_t,\ 
 \tilde{B}^H_t=B^H_t
 -\int_0^t b(B^H_s+\phi_s,\mathcal{L}^P_{X_s})
 -K^H(t,s)\dot{\phi_s}ds.$
 	By the expression \eqref{kh int}, 
 	we can obtain
 		\begin{equation*}
 			\tilde{W}_t=W_t-\int_0^t (K^H)^{-1}(\int_0^\cdot b(B^H_r+\phi_r,\mathcal{L}^P_{X_r} )dr)(s)-\dot{\phi}_sds.
 		\end{equation*}
 	Denote 
 	\begin{equation*}
 		\int_0^t u_sds=\int_0^t b(B^H_s+\phi_s,\mathcal{L}^P_{X_s} )ds-(\phi_t-x).
 	\end{equation*}
 	To apply Theorem \ref{Girsanov}, we first verify that \( \int_0^s b(B^H_r + \phi_r, \mathcal{L}^P_{X_r})  dr \in I_{0^+}^{H + \frac{1}{2}}(L^2([0,1])) \). This allows us to conclude that
 	\( \int_0^s b(B^H_r + \phi_r, \mathcal{L}^P_{X_r})  dr -(\phi_s-x)\in I_{0^+}^{H + \frac{1}{2}}(L^2([0,1])) \).

 	\begin{lemma}
 		In both the regular and singular cases, we have that
 		\begin{equation*}
 			\int_0^s b(B^H_r + \phi_r, \mathcal{L}^P_{X_r})  dr \in I_{0^+}^{H + \frac{1}{2}}(L^2([0,1])).
	\end{equation*}
 	\end{lemma}
 	\begin{proof}
 		For $H \leq \frac{1}{2}$, this condition is evidently satisfied.  

For \( H > \frac{1}{2} \), in view of \eqref{Weil}, it suffices to demonstrate that \( b(B^H_r + \phi_r, \mathcal{L}^P_{X_r}) \) is Hölder continuous with an exponent strictly greater than \( H - \frac{1}{2} \).
 Since $\phi$ is $H$-Hölder continuous and $B^H$ is $(H - \varepsilon)$-Hölder continuous almost surely for any $\varepsilon > 0$, we have  
\begin{align*}
    &|b(B^H_t + \phi_t, \mathcal{L}^P_{X_t}) - b(B^H_s + \phi_s, \mathcal{L}^P_{X_s})| \\
    \leq &L_1\big(|B^H_t + \phi_t - B^H_s - \phi_s| + [\mathbb{E}_P(X_t-X_s)^2]^{\frac{1}{2}}\big) \\  
    \leq& L_1\big(\Vert B^H \Vert_{H-\varepsilon} |t - s|^{H-\varepsilon} + \Vert \phi \Vert_{H} |t - s|^{H} +\Vert b\Vert_\infty |t-s| +(t - s)^H\big).  
\end{align*} 
Therefore, $b(B^H_r + \phi_r, \mathcal{L}^P_{X_r})$ possesses the Hölder continuity of order strictly greater than $H - \frac{1}{2}$, which implies that  
\[
\int_0^s b(B^H_r + \phi_r, \mathcal{L}^P_{X_r}) \, dr \in I_{0^+}^{H + \frac{1}{2}}(L^2([0,1])).
\]
\end{proof}

 	Applying Theorem \ref{Girsanov}, we have that
 	$\tilde{W}$ is a standard Brownian motion
 	under another probability measure  $\tilde{P}$ and the Radon-Nikodym derivative satisfies
 	\begin{align*}
 		\xi_1&=\frac{d\tilde{P}}{dP}\\
 		&=\exp\left( -\int_0^1 (K^H)^{-1} \left( \int_0^s u_r dr \right) dW_s - \frac{1}{2} \int_0^1 \left[ (K^H)^{-1} \left( \int_0^s u_r  dr \right)\right]^2  ds \right).  
 			\end{align*}
 	The	proof of $\mathbb{E}(\xi_1)=1$ can be referenced in \cite[Lemma 10]
{Nualart} .
 	Using \eqref{kh exist}, we have
 	\begin{align*}
 		Y_t &=\phi_t+\int_0^t b(B^H_s+\phi_s,\mathcal{L}^P_{X_s} )
 		-K^H(t,s)\dot{\phi_s}ds+\tilde{B}^H_t\\
 		&=x+\int_0^t b(Y_s,\mathcal{L}^P_{X_s})ds+\tilde{B}^H_t.
 	\end{align*}
 	 Now there are two  weak solutions $(X, B^H)$ and $(Y, \tilde{B}^H)$ defined on distinct  probability spaces $(\Omega,\mathcal{F},P)$ and $(\Omega,\mathcal{F},\tilde{P})$ respectively. These solutions satisfy the following  stochastic differential equation driven by fractional Brownian motion:

\begin{equation*}
Z_t = x + \int_0^t b(Z_s, \mathcal{L}^P_{X_s})ds + W^H_t,
\end{equation*}
where  $W^H$ represents the driving noise ($B^H$ in the original probability space and $\tilde{B}^H$ in the tilted space).  
 	So the ratio can be simplified to 
 		\begin{align}
 	\notag	P(\Vert X-\phi\Vert\leq \varepsilon )
 		&=\tilde{P}(\Vert Y-\phi\Vert\leq \varepsilon )=\tilde{P}(\Vert B^H \Vert\leq\varepsilon )
 		=\int_\Omega I_{\Vert B^H \Vert\leq\varepsilon}d\tilde{P}\\
 	\notag	&=\int_\Omega I_{\Vert B^H\Vert\leq\varepsilon}\xi_1 dP
 		=\mathbb{E}(\xi_1I_{\Vert B^H\Vert\leq\varepsilon})\\
 	\notag	&=\mathbb{E}\Bigg[\exp\Bigg( -\int_0^1 (K^H)^{-1} \big( \int_0^s u_s  dr \big) dW_s \\ 
 	&\quad\quad- \frac{1}{2} \int_0^1 \big[ (K^H)^{-1} \big( \int_0^s u_r  dr \big) \big]^2 \, dr \Bigg)I_{\Vert B^H\Vert\leq \varepsilon}\Bigg].\label{Gir}
 		\end{align}

 	For simplicity,  \( \mathcal{L}_{X} \) will continue to be used to denote \( \mathcal{L}^P_{X} \) throughout this paper.

	Next, we will demonstrate that in both the regular and singular cases, the Onsager-Machlup functional can be represented in the following form:
\begin{equation}\label{jphi}
J(\phi, \dot{\phi}) =-\frac{1}{2}
\int_0^1  \left(\dot{\phi}_s - (K^H)^{-1} \int_0^s b(\phi_u, \mathcal{L}_{X_u})  du \right)^2 + d_H b_x(\phi_s, \mathcal{L}^P_{X_s})  ds.
\end{equation}
 	\begin{remark}
 		The specific expressions of \( J(\phi, \dot{\phi}) \) in the singular and regular cases differ because \( (K^H)^{-1} \) has different representations in these cases. As shown in   \eqref{g1/2} and \eqref{l1/2}, the specific expression is  
\[
(K^H)^{-1} \int_0^s b(\phi_u, \mathcal{L}_{X_u}) \, du =
\begin{cases}
s^{-\alpha} I_{0^+}^\alpha s^\alpha b(\phi_s, \mathcal{L}_{X_s}), \quad H \leq 1/2, \\
s^{\alpha} D_{0^+}^\alpha s^{-\alpha} b(\phi_s, \mathcal{L}_{X_s}), \quad H \geq 1/2.
\end{cases}
\]  
\end{remark}
 	\begin{remark}
 	When the drift term does not depend on the distribution, our result simplifies to  
\begin{equation*}  
J(\phi, \dot{\phi}) =-\frac{1}{2}
\int_0^1  \left(\dot{\phi}_s - (K^H)^{-1}\left( \int_0^\cdot b(\phi_u)  du\right)(s) \right)^2 + d_H b'(\phi_s)  ds,
\end{equation*}  
which coincides with the classical fractional Onsager-Machlup functional (see \cite{Nualart}).

 	\end{remark}

\subsection{Singular case}
	In this section we will compute the Onsager-Machlup functional for $\frac{1}{4}< H<\frac{1}{2}.$
	\begin{theorem}\label{sincase}
		Let $X$ be the solution of  equation \eqref{fir}. Suppose that $\phi$ is a function such that $\phi-x \in \mathcal{H}^p $ with $p>\frac{1}{H}$, and let $b$ satisfy assumption  \textnormal{(\textbf{Hyp 1})} and $\frac{1}{4}< H<\frac{1}{2}.$ Then the Onsager-Machlup functional of $X$ for 
		norms $\Vert \cdot\Vert _\beta$ with $0<\beta<H-\frac{1}{4}$ and $\Vert \cdot\Vert _\infty$
		can be expressed as follows:
		\begin{equation}
			J(\phi,\dot{\phi})=-\frac{1}{2}
			\int_0^1  (\dot{\phi}_s-s^{-\alpha} I_{0^+}^\alpha s^\alpha b(\phi_s,\mathcal{L}_{X_s} ))^2+d_H b_x(\phi_s,\mathcal{L}_{X_s})ds,\label{oms}
 		\end{equation}
	where 
	\begin{equation*}
		d_H=\left(\frac{2H\Gamma(\frac{3}{2}-H)\Gamma(H+\frac{1}{2})}{\Gamma(2-2H)}\right)^\frac{1}{2},
	\end{equation*}
	$\alpha=\frac{1}{2}-H$ and $K^H\dot{\phi}=\phi-x.$
	\end{theorem}
	\begin{proof}
	We first prove the case of the Hölder norm. Using \eqref{l1/2} and \eqref{Gir}, we have
\begin{align*}
P(\Vert X-\phi\Vert_\beta \leq \varepsilon)
&=\mathbb{E}\Big(\exp\Big[ -\int_0^1 (K^H)^{-1} \left( \int_0^s u_r  dr \right) dW_s \\
 &\qquad\qquad\quad- \frac{1}{2} \int_0^1 \left( (K^H)^{-1} \left( \int_0^s u_r  dr \right) \right)^2 ds\Big] I_{\Vert B^H\Vert_\beta \leq \varepsilon}\Big) \\
&=\mathbb{E}\Bigg(\exp\Bigg[\int_0^1 \big(s^{-\alpha}I_{0^+}^\alpha s^\alpha b(B^H_s+\phi_s, \mathcal{L}_{X_s}) - \dot{\phi}_s\big) dW_s \\
&\qquad\qquad\quad - \frac{1}{2}\int_0^1 \big(s^{-\alpha}I_{0^+}^\alpha s^\alpha b(B^H_s+\phi_s, \mathcal{L}_{X_s}) - \dot{\phi}_s\big)^2 ds\Bigg] I_{\Vert B^H\Vert_\beta \leq \varepsilon}\Bigg) \\
&=\mathbb{E}\Big(\exp(A_1 + A_2 + A_3 + A_4) I_{\Vert B^H\Vert_{\beta} \leq \varepsilon}\Big) \\
&\quad \cdot \exp\Bigg(-\frac{1}{2}\int_0^1 \big(\dot{\phi}_s - s^{-\alpha}I_{0^+}^\alpha s^\alpha b(\phi_s, \mathcal{L}_{X_s})\big)^2 ds\Bigg),
\end{align*}

 where
 \begin{align*}
 	A_1 &= \int_0^1 s^{-\alpha} I_{0^+}^\alpha s^\alpha b(B^H_s+\phi_s, \mathcal{L}_{X_s}) dW_s,\\
 	A_2 &= -\int_0^1 \dot{\phi}_s dW_s,\\
	A_3 &= -\frac{1}{2} \int_0^1  \big(s^{-\alpha} I_{0^+}^\alpha s^\alpha b(B^H_s+\phi_s, \mathcal{L}_{X_s})\big)^2 - \big(s^{-\alpha} I_{0^+}^\alpha s^\alpha b(\phi_s, \mathcal{L}_{X_s})\big)^2  ds,   \\
	A_4 &= \int_0^1 \dot{\phi}_s s^{-\alpha} I_{0^+}^\alpha s^\alpha \left( b(B^H_s+\phi_s, \mathcal{L}_{X_s}) - b(\phi_s, \mathcal{L}_{X_s}) \right) ds.
 \end{align*}

	\MakeUppercase{\romannumeral 1}.Term $A_2$:
	
	For any $c\in \mathbb{R},$ using Theorem $\ref{thm2}$, we obtain
	\begin{equation}\label{as2}
		\lim_{\varepsilon\to 0}\mathbb{E}(\exp(cA_2)\mid \Vert B^H\Vert_\beta \leq \varepsilon)=1.
	\end{equation}
		
	\MakeUppercase{\romannumeral 2}.Term $A_4$:

	According to that \(b\) is Lipschitz continuous with constant \(L_1\), under the condition $\Vert B^H\Vert_\beta\leq \varepsilon$, we have
\begin{align}
\notag &\left| s^{-\alpha} I_{0^+}^\alpha s^\alpha \left( b(B^H_s + \phi_s, \mathcal{L}_{X_s}) - b(\phi_s, \mathcal{L}_{X_s}) \right) \right| \\
\notag =& \frac{1}{\Gamma(\alpha)} s^{-\alpha}
\left| \int_0^s r^\alpha (s-r)^{\alpha-1} \left( b(B^H_r + \phi_r, \mathcal{L}_{X_r}) - b(\phi_r, \mathcal{L}_{X_r}) \right) dr \right| \\
\notag \leq & \frac{L_1}{\Gamma(\alpha)} s^{-\alpha}
\int_0^s r^\alpha (s-r)^{\alpha-1} |B^H_r| dr \\
\notag \leq & \frac{L_1 \varepsilon}{\Gamma(\alpha)} s^{-\alpha}
\int_0^s r^\alpha (s-r)^{\alpha-1} dr \\
\notag \leq &\frac{L_1 \varepsilon}{\Gamma(\alpha)} s^{\alpha}
\int_0^1 x^\alpha (1-x)^{\alpha-1} dx \\
= &\frac{\beta(1+\alpha, \alpha)}{\Gamma(\alpha)} L_1 s^\alpha \varepsilon .\label{a416}
\end{align}
	From \eqref{a416}, it is easy to obtain the estimate for the term	$A_4$:
\begin{align*}
 	|A_4| &\leq\int_0^1 |\dot{\phi}_s| |s^{-\alpha} I_{0^+}^\alpha s^\alpha \left( b(B^H_s+\phi_s, \mathcal{L}_{X_s}) - b(\phi_s, \mathcal{L}_{X_s}) \right)|ds \\
  		&\leq \frac{\beta(1+\alpha, \alpha)}{\Gamma(\alpha)} L_1 \varepsilon\int_0^1 |\dot{\phi_s}|s^\alpha ds \\
  		&\leq \frac{\beta(1+\alpha, \alpha)}{\Gamma(\alpha)} L_1\Vert \dot{\phi}\Vert_{L^p} \varepsilon.
 \end{align*}
 	Therefore, we conclude
 	\begin{equation}
 		\lim_{\varepsilon \to 0} \mathbb{E} \left( \exp(A_4) \mid \Vert B^H \Vert_\beta \leq \varepsilon \right) \\
= 1. \label{as4}
\end{equation}

 \MakeUppercase{\romannumeral 3}.
 Term $A_3$:
 
 We first derive an estimate similar to \eqref{a416} under the condition $\Vert B^H\Vert_\beta\leq \varepsilon$:
 \begin{align}
\notag &\left| s^{-\alpha} I_{0^+}^\alpha s^\alpha \left( b(B^H_s + \phi_s, \mathcal{L}_{X_s}) +b(\phi_s, \mathcal{L}_{X_s}) \right) \right| \\
\notag =& \frac{1}{\Gamma(\alpha)} s^{-\alpha}
\left| \int_0^s r^\alpha (s-r)^{\alpha-1} \left( b(B^H_r + \phi_r, \mathcal{L}_{X_r}) + b(\phi_r, \mathcal{L}_{X_r}) \right) dr \right| \\
\notag \leq & \frac{2\Vert b\Vert_\infty }{\Gamma(\alpha)} s^{-\alpha}
\int_0^s r^\alpha (s-r)^{\alpha-1}  dr \\
\notag \leq & \frac{2\Vert b\Vert_\infty  }{\Gamma(\alpha)} s^{-\alpha}
\int_0^s r^\alpha (s-r)^{\alpha-1} dr \\
\notag \leq &\frac{2\Vert b\Vert_\infty  }{\Gamma(\alpha)} s^{\alpha}
\int_0^1 x^\alpha (1-x)^{\alpha-1} dx \\
= &\frac{\beta(1+\alpha, \alpha)}{\Gamma(\alpha)} 2\Vert b\Vert_\infty   \label{a316}.
\end{align}
Using \eqref{a416} and \eqref{a316}, we can obtain the estimate for $A_3$:
\begin{align*}
|A_3| &\leq \frac{1}{2}\int_0^1 \left| s^{-\alpha} I_{0^+}^\alpha s^\alpha \left( b(B^H_s + \phi_s, \mathcal{L}_{X_s}) -b(\phi_s, \mathcal{L}_{X_s}) \right) \right| \\
&\quad\quad\quad\ \cdot\left| s^{-\alpha} I_{0^+}^\alpha s^\alpha \left( b(B^H_s + \phi_s, \mathcal{L}_{X_s}) +b(\phi_s, \mathcal{L}_{X_s}) \right) \right|  ds \\
&\leq \frac{\beta(1+\alpha, \alpha)^2}{\Gamma(\alpha)^2} L_1\Vert \dot{\phi}\Vert_{L^p} \Vert b\Vert_\infty\varepsilon.
\end{align*}
	Therefore,
	\begin{equation}
		\lim_{\varepsilon \to 0} \mathbb{E}\big(\exp(A_3) \mid \Vert B^H \Vert_\beta \leq \varepsilon\big) \\
        = 1.\label{as3}
	\end{equation}

	\MakeUppercase{\romannumeral 4}.Term $A_1$:

	Based on the assumption that $b$ satisfies hypothesis (\textbf{Hyp 1}), we can expand  $b(B^H_s+\phi_s,\mathcal{L}_{X_s})$ as
\[
b(B^H_s+\phi_s,\mathcal{L}_{X_s}) = b(\phi_s,\mathcal{L}_{X_s}) + b_x(\phi_s,\mathcal{L}_{X_s})B^H_s + R_s,
\]  
where the remainder term \( R_s \) satisfies  
\begin{equation}\label{Rs}
	\sup_{0 \leq s \leq 1} |R_s| \leq k \varepsilon^2,
\end{equation}
provided that \( \Vert B^H \Vert_\beta \leq \varepsilon \).
Hence, we can decompose \( A_1 \) into three parts:
\begin{align*}
	A_1 &= \int_0^1 s^{-\alpha} I_{0^+}^\alpha s^\alpha  b(B^H_s+\phi_s,\mathcal{L}_{X_s}) \, dW_s \\
	&= \int_0^1 s^{-\alpha} I_{0^+}^\alpha s^\alpha \big(b(\phi_s,\mathcal{L}_{X_s}) + b_x(\phi_s,\mathcal{L}_{X_s})B^H_s + R_s\big) \, dW_s \\
	&= C_1 + C_2 + C_3,
\end{align*}
where
\begin{align*}
C_1 &= \int_0^1 s^{-\alpha} I_{0^+}^\alpha s^\alpha  b(\phi_s,\mathcal{L}_{X_s}) \, dW_s, \\
C_2 &= \int_0^1 s^{-\alpha} I_{0^+}^\alpha s^\alpha  \big(b_x(\phi_s,\mathcal{L}_{X_s}) B^H_s\big) \, dW_s, \\
C_3 &= \int_0^1 s^{-\alpha} I_{0^+}^\alpha s^\alpha  R_s \, dW_s.
\end{align*}

For the term \( C_1 \), it can be directly obtained using Theorem \ref{thm2} that
\[
\lim_{\varepsilon \to 0} \mathbb{E}\big(\exp(cC_1) \mid \| B^H \|_\beta \leq \varepsilon\big) = 1, \quad \forall c \in \mathbb{R}.
\]

	We can represent $C_2$ in the form of a double stochastic integral. Applying the	representation \eqref{Weil}, we obtain that
	\begin{align*}
		\Gamma(1-\alpha)C_2 &=\Gamma(1-\alpha) \int_0^1 s^{-\alpha} I_{0^+}^\alpha s^\alpha (b_x(\phi_s,\mathcal{L}_{X_s} )B^H_s) dW_s\\
		& =\frac{1}{\Gamma(\alpha)} \int_0^1
		s^{-\alpha}\int_0^s r^\alpha b_x(\phi_r,\mathcal{L}_{X_r} )B^H_r(s-r)^{\alpha-1}drdW_s\\
		&=\frac{1}{\Gamma(\alpha)} \int_0^1
		s^{-\alpha}\int_0^s r^\alpha b_x(\phi_r,\mathcal{L}_{X_r} )(s-r)^{\alpha-1} \int_0^r K^H(r,u)dW_u dr dW_s.
		 	\end{align*}
	
	Using \eqref{kh} we obtain 
			\begin{align*}
			|K^H(r,u)|&\leq c_H((r-u)^{-\alpha}+\int_u^r(\theta-u)^{-\alpha-1}\left(1-\left(\frac{u}{\theta} \right)^\alpha \right)d\theta )\\
			& \leq c_H\left((r-u)^{-\alpha}+u^{-\alpha}\int_u^r(\frac{\theta}{u}-1)^{-\alpha-1}\left(1-\left(\frac{\theta}{u} \right)^{-\alpha} \right)d\frac{\theta}{u}\right)\\
			&\leq  c_H\left((r-u)^{-\alpha}+u^{-\alpha}\int_1^\infty (x-1)^{-\alpha-1}(1-x^{-\alpha})dx\right)\\
			&\leq \tilde{k}((r-u)^{-\alpha}+u^{-\alpha}),
	\end{align*}
	where $\tilde{k}= c_H+\int_1^\infty (x-1)^{-\alpha-1}(1-x^{-\alpha})dx$. Therefore,
		\begin{align}
		\notag	&\int_0^s r^\alpha b_x(\phi_s,\mathcal{L}_{X_s} )(s-r)^{\alpha-1}\left( \int_0^r K^H(r,u)^2 du\right)^\frac{1}{2}dr\\
		\notag	&\leq \tilde{k} \int_0^s r^\alpha \Vert b_x\Vert_\infty (s-r)^{\alpha-1}\left( \int_0^r  ((r-u)^{-2\alpha}+u^{-2\alpha})du\right)^\frac{1}{2}dr\\
			&\leq \tilde{k}\frac{2}{1-2\alpha}\Vert b_x\Vert_\infty \int_0^s r^{1-\alpha}  (s-r)^{\alpha-1}dr<\infty.\label{fbn1}
		\end{align}

		According to Theorem \ref{fubini} and \eqref{fbn1}, we have 
	 \begin{equation*}
	 	C_2=\int_0^1\int_0^s f(s,u)dW_udW_s,
	 \end{equation*}
	 where
	 \begin{align*}
	 	f(s,u)=&\frac{s^{-\alpha}}{\Gamma(\alpha)}
	 	\int_u^s r^\alpha b_x(\phi_r,\mathcal{L}_{X_r} )(s-r)^{\alpha-1} K^H(r,u)dr.
	 \end{align*}
	 From the expression, it is evident that \( f(s,u) \) is undefined when \( s = u \). However, we can investigate the limit as \( u \to s \) to ensure that \( f(s,u) \) becomes a continuous function on \( [0,1] \times [0,s] \).
	 
	 Using the expression of the kernel $K^H$, we obtain
\begin{align*}
f(s, u) &= \frac{c_H}{\Gamma(\alpha)} s^{-\alpha}
\int_u^s r^\alpha b_x(\phi_r,\mathcal{L}_{x_r}) (s - r)^{\alpha-1}(r - u)^{-\alpha} dr \\
&\quad + \frac{c_H \left(\frac{1}{2} - H \right)}{\Gamma(\alpha)}  s^{-\alpha}
\int_u^s \int_u^r r^\alpha b_x(\phi_r,\mathcal{L}_{x_r}) (s - r)^{\alpha-1} \\
&\quad\quad\quad\quad \qquad\qquad\qquad\quad\quad\cdot (\theta - u)^{-\alpha-1} \left(1 - \left(\frac{u}{\theta}\right)^\alpha\right) d\theta dr \\
&= (D_1 + D_2)(s, u),
\end{align*}
	 where
	 \begin{equation*}
	 	D_1(s,u)= \frac{c_H}{\Gamma(\alpha)} s^{-\alpha}
\int_u^s r^\alpha b_x(\phi_r,\mathcal{L}_{X_r} )(s - r)^{\alpha-1}(r - u)^{-\alpha} dr 
	 \end{equation*}
	 and
	 \begin{align*}
	 	D_2(s,u)&=\frac{c_H \left(\frac{1}{2} - H \right)}{\Gamma(\alpha)}  s^{-\alpha}
\int_u^s \int_u^r r^\alpha b_x(\phi_r,\mathcal{L}_{X_r} )(s - r)^{\alpha-1}\\
&\quad\quad\qquad\qquad\qquad\qquad\quad \cdot (\theta - u)^{-\alpha-1} \left(1 - \left(\frac{u}{\theta}\right)^\alpha\right) d\theta dr.
	 \end{align*}
	The change of variable $x=\frac{r-u}{s-u}$ yields that
	\begin{align*}
		D_1(s,u)&= \frac{c_H}{\Gamma(\alpha)} s^{-\alpha} \int_0^1 ((s-u)x+u)^\alpha 
		b_x(\phi_{(s-u)x+u},\mathcal{L}_{X_{(s-u)x+u}} )\\
		&\quad \quad\quad\qquad\quad\cdot(1-x)^{\alpha-1}x^{-\alpha}dx.
	\end{align*}	 
	 Letting $u\to s$ we have
	\begin{align*}
D_1(s, s) &= \frac{c_H}{\Gamma(\alpha)} b_x(\phi_s,\mathcal{L}_{X_s} )
\int_0^1 (1 - x)^{\alpha-1}x^{-\alpha} dx \\
&= c_H \, \Gamma(1 - \alpha) b_x(\phi_s,\mathcal{L}_{X_s} ).
\end{align*}
	 For the term \( D_2 \), by making a variable substitution \( v = \frac{\theta-u}{r-u}\), we obtain that
	 \begin{equation*}
	 	D_2(s, u) =    \frac{c_H \alpha}{\Gamma(\alpha)}  s^{-\alpha} \int_u^s r^\alpha b_x(\phi_r,\mathcal{L}_{X_r} ) (s - r)^{\alpha-1} (r - u)^{-\alpha} B(u,r)  dr,
	 \end{equation*}
	 where
	 \begin{equation*}
	 	B(u,r)=\int_0^1 v^{-\alpha-1} \left( 1 - \left( \frac{u}{(r-u)v+u} \right)^{\alpha} \right) dv.
	 \end{equation*}
	 Letting $x=\frac{r-u}{s-u}$, we have 	
	 \begin{align*}
	 	D_2(s, u) = &   \frac{c_H \alpha}{\Gamma(\alpha)}  s^{-\alpha} \int_0^1 ((s-u)x+u)^\alpha b_x(\phi_{(s-u)x+u},\mathcal{L}_{X_{(s-u)x+u}} )\\
	 	&\qquad\qquad\quad\cdot (1-x)^{\alpha-1} x^{-\alpha} B(u,(s-u)x+u)  dx,
	 \end{align*}
	 and then 
	 \begin{equation*}
	 	D_2(s, s) =c_H \alpha \Gamma(1-\alpha)b_x(\phi_{s},\mathcal{L}_{X_{s}} )B(s,s)=0.
	 \end{equation*}
	 Therefore, $f(s,u)$ is continuous on $[0,1] \times [0,s]$ and
	 \begin{equation*}
	 	f(s,s)=c_H \, \Gamma(1 - \alpha) b_x(\phi_s,\mathcal{L}_{X_s} ).
	 \end{equation*}
	 
	 Define    
	 \begin{equation*}
	 	\tilde{f}(s,u)=\frac{f(s,u)I_{u\leq s}+f(u,s)I_{s<u}}{2},
	 \end{equation*}
	 where $\tilde{f}(s,u)$ is a continuous and symmetric function on $[0,1]\times [0,1]$ and  
	 \begin{align*}
	 	\int_0^1\int_0^1 \tilde{f}(s,u) dW_u dW_s&=\int_0^1\int_0^s \frac{\tilde{f}(s,u)}{2} dW_u dW_s
	 	+\int_0^1\int_s^1 \frac{\tilde{f}(u,s)}{2} dW_udW_s \\
	 	&=\int_0^1\int_0^s \frac{\tilde{f}(s,u)}{2} dW_u dW_s
		+\int_0^1\int_0^u \frac{\tilde{f}(u,s)}{2} dW_sdW_u\\
		&=C_2.
	 \end{align*}
	Proof that $K(\tilde{f})$ is nuclear can be referenced in \cite[Lemma 13]{Nualart}. According to Theorem \ref{thm3}, the exponential condition limit of $C_2$ can be obtained from the trace of $\tilde{f}(s,u)$ :
	 \begin{align*}
	 		\lim_{\varepsilon\to 0} \mathbb{E}(\exp(cC_2)|\Vert B^H\Vert_\beta \leq \varepsilon)&=\exp(-c\operatorname{Tr} \tilde{f})\\
	 		&=\exp\left(-\int_0^1 c\tilde{f}(s,s)ds\right)\\&=\exp\left(-\int_0^1 \frac{c}{2} f(s,s)ds\right)\\
	 		&= \exp\left(-\frac{cd_H}{2}\int_0^1 b_x(\phi_{s},\mathcal{L}_{X_{s}}) \, ds\right) .
	 \end{align*}
	 
	It only remains to study the behaviour of the term \(C_3\). For any \(c \in \mathbb{R}\), we define:  
\[
M_t = c \int_0^t s^{-\alpha} I_{0^+}^\alpha s^\alpha R_s \, dW_s.
\]
	 Then $M_t$ is a martingale. 
	Using \eqref{Rs}, we obtain the following estimate:
	\begin{align*}
		\langle M \rangle_t &= c^2 \int_0^t (s^{-\alpha} I_{0^+}^\alpha s^\alpha R_s)^2 \, ds \\
		&\leq c^2 \int_0^1 (s^{-\alpha}\frac{1}{\Gamma(\alpha)}\int_0^s (s-r)^{\alpha-1}r^\alpha R_rdr )^2ds\\
		&\leq \frac{c^2}{\Gamma(\alpha)^2}k^2\varepsilon^4 \int_0 ^1 \frac{s^\alpha}{\alpha}ds \\
		&\leq k' \varepsilon^4,
		\end{align*}
	where $k'=\frac{c^2k^2}{\Gamma(\alpha)^2\alpha(\alpha+1)}$.

 		Applying the  exponential inequality for martingales, we have 
 		\begin{equation*}
 			P\left(\left|\int_0^1 s^{\alpha}D_{0^+}^\alpha s^{-\alpha} R_sdW_s\right|> \xi, \Vert B^H\Vert _\beta \leq \varepsilon\right)
 	\leq \exp\left(-\frac{\xi^2}{2k'\varepsilon^4}\right).
	\end{equation*}
 	Combining with the small-ball behavior of fractional Brownian motion under the Hölder norm \eqref{sbb holder}, we obtain:
\begin{equation}\label{p}
P\left( \left| \int_0^1 s^{\alpha} D_{0^+}^\alpha s^{-\alpha} R_s dW_s \right| > \xi \mid \Vert B^H \Vert_\beta \leq \varepsilon \right)
\leq \exp\left( -\frac{\xi^2}{2k' \varepsilon^4} \right) \exp\left( K_2 \varepsilon^{\frac{-1}{H-\beta}} \right).
\end{equation}

			Under the condition $\Vert B^H\Vert \leq \varepsilon$, let $p_{\varepsilon,t}$ be the measure on $\mathbb{R}$ induced by $M_t$, and let $F_{\varepsilon,t}$ be the distribution function of $M_t$. 	
	Hence, for any $\delta>0,$ by \eqref{p2} we have 
		\begin{align*}
			&\mathbb{E}(\exp(cC_3)\mid\Vert B^H\Vert_\beta\leq\varepsilon)\\
			=&\mathbb{E}(\exp(M_t)\mid\Vert B^H\Vert_\beta\leq\varepsilon)\\
 		=&\int_{-\infty}^\infty e^\xi p_{\varepsilon,t}(d\xi)\\
 		=&\int_{-\infty}^{\delta} e^\xi p_{\varepsilon,t}(d\xi) +\int_{\delta}^{+\infty} e^\xi  p_{\varepsilon,t}(d\xi)        \\
 		\leq & e^{\delta}- e^\xi(1-F(\xi))\Big|^\infty_\delta 
 		+\int_{\delta}^{+\infty} e^\xi (1-F(\xi)) d\xi \\
 		  \leq & e^{\delta}+\int_{\delta}^{+\infty} e^\xi P\left(\left|c\int_0^1s^{-\alpha}I_{0+}^\alpha s^\alpha R_sdW_s\right|>\xi\mid\Vert B^H\Vert_\beta \leq \varepsilon\right) d\xi\\
 		 &+e^{\delta}P\left(\left|c\int_0^1s^{-\alpha}I_{0+}^\alpha s^\alpha R_sdW_s\right|>\delta \mid\Vert B^H\Vert_\beta \leq \varepsilon\right)\\
 		 \leq & e^\delta +e^\delta   \exp\left(-\frac{\delta^2}{2k'\varepsilon^4}\right)\exp(K_2\varepsilon^\frac{-1} {H-\beta})+ \int_{\delta}^{+\infty} e^\xi\exp\left(-\frac{\xi^2}{2k'\varepsilon^4}\right)\exp(K_2\varepsilon^\frac{-1} {H-\beta})d\xi \\
			\leq & e^\delta +e^\delta   \exp\left(-\frac{\delta^2}{2k'\varepsilon^4}\right)\exp(K_2\varepsilon^\frac{-1} {H-\beta})+\exp\left(\frac{k'\varepsilon^4}{2}+K_2\varepsilon^{-\frac{1}{H-\beta}}\right) \int_{\frac{\delta}{\sqrt{2k'}\varepsilon^2}-\frac{\sqrt{2k'}\varepsilon^2}{2}}^{+\infty} e^{-x^2}dx\\
			\leq & e^\delta +e^\delta   \exp\left(-\frac{\delta^2}{2k'\varepsilon^4}\right)\exp(K_2\varepsilon^\frac{-1} {H-\beta})\\
			&+\sqrt{2k'}\frac{\pi}{8}\exp\left(\frac{k'\varepsilon^4}{2}+K_2\varepsilon^{-\frac{1}{H-\beta}}-\left({\frac{\delta}{\sqrt{2k'}\varepsilon^2}-\frac{\sqrt{2k'}\varepsilon^2}{2}}\right)^2\right)\varepsilon^2,
		\end{align*}
	 when $\beta< H-\frac{1}{4}$.
	 
	Thus, when $\beta< H-\frac{1}{4}$, we have
		\begin{align*}
			\lim_{\varepsilon\to 0} \mathbb{E}(\exp(cC_3)\mid\Vert B^H\Vert_\beta\leq\varepsilon)  \leq e^\delta.
		\end{align*}
	By the arbitrariness of $\delta$, it can be concluded that
	\begin{equation*}
		\lim_{\varepsilon\to 0} \mathbb{E}(\exp(cC_3)\mid \Vert B^H\Vert_\beta\leq\varepsilon)\leq 1, \quad \forall c\in \mathbb{R}.
	\end{equation*}
	Hence
	\begin{align}
	\lim_{\varepsilon\to 0} \mathbb{E}(\exp(A_1)\mid\Vert B^H\Vert_\beta\leq\varepsilon)
	&= \lim_{\varepsilon\to 0} \mathbb{E}(\exp(C_1+C_2+C_3)\mid\Vert B^H\Vert_\beta\leq\varepsilon) \notag \\
	&= \exp\left(-\frac{cd_H}{2}\int_0^1 b_x(\phi_{s},\mathcal{L}_{X_{s}}) \, ds\right). \label{as1}
\end{align}

		In summary, based on \eqref{as1}, \eqref{as2}, \eqref{as3}, and \eqref{as4}, and by further applying Theorem \ref{lem8}, we can derive our result:
\begin{align*}
\lim_{\varepsilon\to 0} \frac{P(\Vert X - \phi \Vert_\beta \leq \varepsilon)}{P(\Vert B^H \Vert_\beta \leq \varepsilon)} 
&=\lim_{\varepsilon\to 0} \mathbb{E}\Big( \exp(A_1 + A_2 + A_3 + A_4) \mid {\Vert B^H \Vert_{\beta} \leq \varepsilon} \Big) \\
&\quad \cdot \exp\Bigg( -\frac{1}{2} \int_0^1 \left( \dot{\phi}_s - s^{-\alpha} I_{0^+}^\alpha s^\alpha b(\phi_s, \mathcal{L}_{X_s}) \right)^2 \, ds \Bigg) \\
&= \exp\left( -\frac{1}{2}\int_0^1  ( \dot{\phi}_s - s^{-\alpha} I_{0^+}^\alpha s^\alpha b(\phi_s, \mathcal{L}_{X_s}) )^2 + d_H b_x(\phi_s, \mathcal{L}_{X_s})  ds \right)\\
&=\exp(J(\phi,\dot{\phi})).
\end{align*}
Finally, we obtain 
\begin{equation*}
	J(\phi,\dot{\phi})=-\frac{1}{2}
			\int_0^1  (\dot{\phi}_s-s^{-\alpha} I_{0^+}^\alpha s^\alpha b(\phi_s,\mathcal{L}_{X_s} ))^2+d_H b_x(\phi_s,\mathcal{L}_{X_s})ds,
\end{equation*}
where 
	\begin{equation*}
		d_H=\left(\frac{2H\Gamma(\frac{3}{2}-H)\Gamma(H+\frac{1}{2})}{\Gamma(2-2H)}\right)^\frac{1}{2}.
	\end{equation*}

The proof for the supremum norm is essentially the same; we only need to replace \eqref{sbb holder} with \eqref{sup} in the proof of \eqref{p}.

	\end{proof}

\subsection{Regular case}
In this section we will compute the Onsager-Machlup functional for $ H>\frac{1}{2}.$

	\begin{theorem}\label{recase}
		Let $X$ be the solution of  equation \eqref{fir}. Suppose that $\phi$ is a function such that $\phi-x \in \mathcal{H}^2 $, and let $b$ satisfy assumption \textnormal{(\textbf{Hyp 2})} and $H>\frac{1}{2}.$ Then the Onsager-Machlup functional of $X$ for the 
		norm $\Vert \cdot\Vert _\beta$ with $H-\frac{1}{2}<\beta<H-\frac{1}{4}$ 
		can be expressed as follows:
		\begin{equation}
			J(\phi,\dot{\phi})=-\frac{1}{2} \int_0^1 (\dot{\phi}_s-s^{\alpha}D_{0^+}^\alpha s^{-\alpha}b(\phi_s,\mathcal{L}_{X_s} ))^2+d_H b_x(\phi_s,\mathcal{L}_{X_s})ds,\label{omr}
		\end{equation}    
	where 
	\begin{equation*}
		d_H=\left(\frac{2H\Gamma(\frac{3}{2}-H)\Gamma(H+\frac{1}{2})}{\Gamma(2-2H)}\right)^\frac{1}{2},
	\end{equation*}
	$\alpha=H-\frac{1}{2}$ and $K^H\dot{\phi}=\phi-x.$
	\end{theorem}

\begin{proof}
				 Using \eqref{g1/2} and \eqref{Gir}, we have
\begin{align*}
P(\Vert X-\phi\Vert_\beta \leq \varepsilon)
&=\mathbb{E}\Big(\exp\Big[ -\int_0^1 (K^H)^{-1} \left( \int_0^s u_r  dr \right) dW_s \\
 &\qquad\qquad\quad- \frac{1}{2} \int_0^1 \left( (K^H)^{-1} \left( \int_0^s u_r  dr \right) \right)^2 ds\Big] I_{\Vert B^H\Vert_\beta \leq \varepsilon}\Big) \\
&= \mathbb{E}\Bigg(\exp\Bigg[\int_0^1 \left(s^{\alpha}D_{0^+}^\alpha s^{-\alpha} b(B^H_s+\phi_s, \mathcal{L}_{X_s}) - \dot{\phi}_s\right) dW_s \\
& \qquad- \frac{1}{2}\int_0^1 \left(s^{\alpha}D_{0^+}^\alpha s^{-\alpha} b(B^H_s+\phi_s, \mathcal{L}_{X_s}) - \dot{\phi}_s\right)^2 ds\Bigg] I_{\Vert B^H\Vert_\beta \leq \varepsilon}\Bigg)  \\
&=\mathbb{E}\left(\exp(A_1 + A_2 + A_3 + A_4) I_{\Vert B^H\Vert \leq \varepsilon}\right) \\
&\quad \cdot \exp\left(-\frac{1}{2}\int_0^1 \left(\dot{\phi}_s - s^{\alpha}D_{0^+}^\alpha s^{-\alpha} b(\phi_s, \mathcal{L}_{X_s})\right)^2 ds\right),
\end{align*}
where
	\begin{align*}
		A_1 &= \int_0^1 s^{\alpha}D_{0^+}^\alpha s^{-\alpha} b(B^H_s+\phi_s, \mathcal{L}_{X_s}) dW_s,\\
		A_2 &= -\int_0^1 \dot{\phi}_s dW_s,\\
		A_3 &= -\frac{1}{2}\int_0^1  (s^{\alpha}D_{0^+}^\alpha s^{-\alpha} b(B^H_s+\phi_s, \mathcal{L}_{X_s}))^2-(s^{\alpha}D_{0^+}^\alpha s^{-\alpha} b(\phi_s, \mathcal{L}_{X_s}))^2  ds,\\
	A_4 &= \int_0^1 \dot{\phi}_s s^{\alpha}D_{0^+}^\alpha s^{-\alpha} \left(b(B^H_s+\phi_s, \mathcal{L}_{X_s}) - b(\phi_s, \mathcal{L}_{X_s})\right) ds.
	\end{align*}

	\MakeUppercase{\romannumeral 1}.Term $A_2$:
	
	For any $c\in \mathbb{R},$ using Theorem \ref{thm2}, we obtain
	\begin{equation}\label{ar2}
		\lim_{\varepsilon\to 0}\mathbb{E}(\exp(cA_2) \mid \Vert B^H\Vert_\beta \leq \varepsilon)=1.
	\end{equation}
		
	\MakeUppercase{\romannumeral 2}.Term $A_4$:
	
	Applying \eqref{Weil}, we have 
		\begin{align*}
 		&s^{\alpha}D_{0^+}^\alpha s^{-\alpha} (b(B^H_s+\phi_s,\mathcal{L}_{X_s} )-b(\phi_s,\mathcal{L}_{X_s} ))\\
 		= &\frac{s^\alpha}{\Gamma(1-\alpha)}(\frac{b(B^H_s+\phi_s,\mathcal{L}_{X_s} )-b(\phi_s,\mathcal{L}_{X
 		_s} )}{s^{2\alpha}}\\
 		& +\alpha \int_0^s \frac{s^{-\alpha}(b(B^H_s+\phi_s,\mathcal{L}_{X_s} )-b(\phi_s,\mathcal{L}_{X_s} ))
 		}
 		{(s-r)^{\alpha+1}}  \\
 		&\quad\quad\quad-\frac{r^{-\alpha}(b(B^H_r+\phi_r,\mathcal{L}_{X_r} )
 		-b(\phi_r,\mathcal{L}_{X_r} ))}{(s-r)^{\alpha+1}}
 		dr) \\ 
 		=& 		 B_1+B_2+B_3,
 	\end{align*}
	where
	\begin{align*}
		B_1 &= \frac{b(B^H_s+\phi_s,\mathcal{L}_{X_s} )-b(\phi_s,\mathcal{L}_{X_s} )}{\Gamma(1-\alpha)s^\alpha},\\
		B_2 &= \frac{\alpha}{\Gamma(1-\alpha)} 
 	\int_0^s \frac{b(B^H_s+\phi_s,\mathcal{L}_{X_s} )-b(\phi_s,\mathcal{L}_{X_s} )}
 	{(s-r)^{\alpha+1}}\\
 	&\quad\qquad\qquad\quad-\frac{b(B^H_r+\phi_r,\mathcal{L}_{X_r} )-b(\phi_r,\mathcal{L}_{X_r} )}{(s-r)^{\alpha+1}} dr, \\
	B_3 &= \frac{\alpha s^\alpha}{\Gamma(1-\alpha)} 
 	\int_0^s \frac{s^{-\alpha}-r^{-\alpha}}{(s-r)^{\alpha+1}}(b(B^H_r+\phi_r,\mathcal{L}_{X_r} )-b(\phi_r,\mathcal{L}_{X_r} ))dr.
	\end{align*}
	Using that $b$ is Lipschitz continuous with a constant $L_1$ and  $\beta>\alpha$, we have
	\begin{align}
	\notag	|B_1| &\leq	\left|\frac{b(B^H_s+\phi_s,\mathcal{L}_{X_s} )-b(\phi_s,\mathcal{L}_{X_s} )}{\Gamma(1-\alpha)s^\alpha}\right|\\
		&\leq \frac{L_1}{\Gamma(1-\alpha)}\varepsilon.\label{B_1'}
	\end{align}
	 	Using the fact that \(b_x\) is Lipschitz continuous with a constant \(L_2\) and bounded, and that \(\phi\) is \(H\)-Hölder continuous, it follows that
		\begin{align*}
    &|b(B^H_s+\phi_s,\mathcal{L}_{X_s}) - b(\phi_s,\mathcal{L}_{X_s}) - 
    \big(b(B^H_r+\phi_r,\mathcal{L}_{X_r}) - b(\phi_r,\mathcal{L}_{X_r})\big)| \\
    \leq &|\int_0^1 b_x(\lambda B^H_s+\phi_s,\mathcal{L}_{X_s})B^H_s-b_x(\lambda B^H_r+\phi_r,\mathcal{L}_{X_r})B^H_r  d\lambda |
     \\
    \leq & |\int_0^1 b_x(\lambda B^H_s+\phi_s,\mathcal{L}_{X_s})(B^H_s - B^H_r)\\
  &\quad +  \big(b_x(\lambda B^H_s+\phi_s,\mathcal{L}_{X_s}) - 
    b_x(\lambda B^H_r+\phi_r,\mathcal{L}_{X_r})\big)B^H_s d\lambda|\\
    \leq & \big(\Vert b_x\Vert_\infty  |s-r|^\beta+L_2(|s-r|^\beta+\Vert \phi\Vert |s-r|^H+2\Vert b\Vert_\infty|s-r|+2|s-r|^H )\big) \varepsilon.
\end{align*}
 Therefore
\begin{align}
   \notag |B_2| &\leq \frac{\alpha}{\Gamma(1-\alpha)} \Big| \int_0^s \frac{1}{(s-r)^{\alpha+1}} \big( b(B^H_s+\phi_s,\mathcal{L}_{X_s}) - 
    b(\phi_s,\mathcal{L}_{X_s})  \\
   \notag &\quad 
  \quad\quad - \big(b(B^H_r+\phi_r,\mathcal{L}_{X_r}) - 
    b(\phi_r,\mathcal{L}_{X_r})\big) \big) dr \Big| \\
      &\leq C' \varepsilon,\label{B_2'}
\end{align}
where
\begin{equation*}
	C'=\frac{\alpha}{\Gamma(1-\alpha)}\left(\frac{\Vert b_x\Vert_\infty }{\beta-\alpha}+L_2\left(\frac{1}{\beta-\alpha}+2\Vert \phi\Vert_H+\frac{2}{1-\alpha}\Vert b\Vert_\infty+4 \right)   \right).
\end{equation*}
 	From the integral  
	\begin{equation}
		\int_0^s \frac{s^{-\alpha}-r^{-\alpha}}{(s-r)^{\alpha+1}}dr = c_\alpha s^{-2\alpha},\label{int}
	\end{equation}
where \( c_\alpha \) is a constant that depends on \(\alpha\), we thus deduce 
	\begin{align}
   \notag |B_3| &\leq \left| \frac{\alpha s^\alpha}{\Gamma(1-\alpha)} 
    \int_0^s \frac{s^{-\alpha}-r^{-\alpha}}{(s-r)^{\alpha+1}} \big(b(B^H_r+\phi_r,\mathcal{L}_{X_r}) - b(\phi_r,\mathcal{L}_{X_r})\big) dr \right| \\ 
  \notag  &\leq \frac{\alpha s^\alpha}{\Gamma(1-\alpha)} 
    \int_0^s \frac{s^{-\alpha}-r^{-\alpha}}{(s-r)^{\alpha+1}} |B^H_r| dr \\ 
    &\leq \frac{c_\alpha M \alpha s^{-\alpha+\beta}}{\Gamma(1-\alpha)} 
    \Vert B^H \Vert_\beta \\
    &\leq \frac{c_\alpha M \alpha}{\Gamma(1-\alpha)} \varepsilon. \label{B_3'}
\end{align}
According to \eqref{B_1'}, \eqref{B_2'} and \eqref{B_3'}, we have 
	\begin{equation}
		|s^{\alpha}D_{0^+}^\alpha s^{-\alpha} (b(B^H_s+\phi_s,\mathcal{L}_{X_s} )-b(\phi_s,\mathcal{L}_{X_s} ))|=|B_1+B_2+B_3|\leq c_B \varepsilon,\label{cb}
	\end{equation}
where $c_B=\frac{L_1}{\Gamma(1-\alpha)}+\frac{c_\alpha M \alpha}{\Gamma(1-\alpha)}+C'.$

Thus, we obtain that
\begin{equation}
	\lim_{\varepsilon \to 0} \mathbb{E}\big( \exp(A_4) \mid \Vert B^H \Vert_\beta \leq \varepsilon \big)=1.\label{ar4}
\end{equation}

	\MakeUppercase{\romannumeral 3}.Term $A_3$ 
	
	We first derive an estimate similar to \eqref{cb} under the condition $\Vert B^H\Vert_\beta\leq \varepsilon$:
	\begin{align}
	\notag	&|s^{\alpha}D_{0^+}^\alpha s^{-\alpha} (b(B^H_s+\phi_s,\mathcal{L}_{X_s} )+b(\phi_s,\mathcal{L}_{X_s} ))|\\
	\notag	= \Big|&\frac{s^\alpha}{\Gamma(1-\alpha)}
	\Bigg(\frac{b(B^H_s+\phi_s,\mathcal{L}_{X_s} )+b(\phi_s,\mathcal{L}_{X
 		_s} )}{s^{2\alpha}}\\
 	\notag	& \qquad\qquad\quad +\alpha \int_0^s \frac{s^{-\alpha}(b(B^H_s+\phi_s,\mathcal{L}_{X_s} )+b(\phi_s,\mathcal{L}_{X_s} ))
 		}
 		{(s-r)^{\alpha+1}}  \\
 	\notag	&\qquad\quad\qquad\qquad\quad-\frac{r^{-\alpha}(b(B^H_r+\phi_r,\mathcal{L}_{X_r} )
 		+b(\phi_r,\mathcal{L}_{X_r} ))}{(s-r)^{\alpha+1}}
 		dr\Bigg) \Big|\\
 		\leq & \frac{2\Vert b\Vert_\infty (1+\alpha c_\alpha)}{\Gamma(1-\alpha)s^\alpha}+\frac{\alpha L_1}{\Gamma(1-\alpha)}\left(\frac{1}{\beta-\alpha}+4\Vert \phi\Vert_H+4\frac{\Vert b\Vert_\infty}{1-\alpha}+4 \right). \label{esa3}
	\end{align}

Using $\eqref{cb}$ and $\eqref{esa3}$	, we can obtain the estimate for $A_3$:
\begin{align*}
|A_3| &\leq \frac{1}{2} \int_0^1 |s^{\alpha} D_{0^+}^\alpha s^{-\alpha} b(B^H_s + \phi_s, \mathcal{L}_{X_s}) 
-  s^{\alpha} D_{0^+}^\alpha s^{-\alpha} b(\phi_s, \mathcal{L}_{X_s})|  \\
   &\quad \quad\quad\cdot |s^{\alpha} D_{0^+}^\alpha s^{-\alpha} b(B^H_s + \phi_s, \mathcal{L}_{X_s}) 
+  s^{\alpha} D_{0^+}^\alpha s^{-\alpha} b(\phi_s, \mathcal{L}_{X_s})| ds\\
&\leq \frac{c_B\Vert b\Vert_\infty (1+\alpha c_\alpha)\varepsilon}{\Gamma(1-\alpha)\alpha}+\frac{\alpha L_1c_B\varepsilon}{2\Gamma(1-\alpha)}\left(\frac{1}{\beta-\alpha}+4\Vert \phi\Vert_H+4\frac{\Vert b\Vert_\infty}{1-\alpha}+4 \right).
\end{align*}
 	Therefore,
 	\begin{equation}
 		\lim_{\varepsilon \to 0} \mathbb{E}\big(\exp(A_3)\mid \Vert B^H \Vert_\beta \leq \varepsilon\big)=1. \label{ar3}
 	\end{equation}

	\MakeUppercase{\romannumeral 4}.Term $A_1$	
	
	Based on the assumption that \( b \) satisfies Hypothesis (\textbf{Hyp 2}), we can expand \( b(B^H_s + \phi_s, \mathcal{L}_{X_s}) \) as follows:
\[
b(B^H_s + \phi_s, \mathcal{L}_{X_s}) = b(\phi_s, \mathcal{L}_{X_s}) + b_x(\phi_s, \mathcal{L}_{X_s}) B^H_s + R_s,
\]
where the remainder term \( R_s \) satisfies
\begin{equation}\label{Rs}
	\sup_{0 \leq s \leq 1} |R_s| \leq \bar{k} \varepsilon^2,
\end{equation}
and \( \bar{k} = \| b_{xx} \|_\infty \), provided that \( \| B^H \|_\beta \leq \varepsilon \).

Hence, we can decompose \( A_1 \) into three parts:
\begin{align*}
	A_1 &= \int_0^1 s^{\alpha} D_{0^+}^\alpha s^{-\alpha} b(B^H_s+\phi_s,\mathcal{L}_{X_s}) \, dW_s \\
	&= \int_0^1 s^{\alpha} D_{0^+}^\alpha s^{-\alpha} \big(b(\phi_s,\mathcal{L}_{X_s}) + b_x(\phi_s,\mathcal{L}_{X_s})B^H_s + R_s\big) \, dW_s \\
	&= C_1 + C_2 + C_3,
\end{align*}
where
\begin{align*}
	C_1 &= \int_0^1 s^{\alpha} D_{0^+}^\alpha s^{-\alpha} b(\phi_s,\mathcal{L}_{X_s}) \, dW_s,\\
	C_2 &= \int_0^1 s^{\alpha} D_{0^+}^\alpha s^{-\alpha} \big(b_x(\phi_s,\mathcal{L}_{X_s}) B^H_s\big) \, dW_s,\\
	C_3 &= \int_0^1 s^{\alpha} D_{0^+}^\alpha s^{-\alpha} R_s \, dW_s.
\end{align*}
Applying the	representation \eqref{Weil}, we obtain
		\begin{align*}
		\Gamma(1-\alpha)C_2 &=\Gamma(1-\alpha) \int_0^1 s^{\alpha}D_{0^+}^\alpha s^{-\alpha} (b_x(\phi_s,\mathcal{L}_{X_s} )B^H_s) dW_s\\
		& =\int_0^1 (\alpha s^{-\alpha} b_x(\phi_s,\mathcal{L}_{X_s} )B^H_s \\
		&\quad+s^\alpha \alpha
		\int_0^s \frac{ s^{-\alpha} b_x(\phi_s,\mathcal{L}_{X_s})
		 -r^{-\alpha} b_x(\phi_r,\mathcal{L}_{X_r})
		 B^H_s}{(s-r)^{\alpha+1}})dW_s\\
		&=\int_0^1  s^{-\alpha}b_x(\phi_s,\mathcal{L}_{X_s})\int_0^s K^H(s,u)dW_udW_s\\
		&\quad+\int_0^1  \alpha s^{\alpha}
		\int_0^s \frac{s^{-\alpha} b_x(\phi_s,\mathcal{L}_{X_s})
		\int_0^s K^H(s,u)dW_u
		}{(s-r)^{\alpha+1}}drdW_s \\
		&\quad-\int_0^1  \alpha s^{\alpha}
		\int_0^s \frac{-r^{-\alpha} b_x(\phi_r,\mathcal{L}_{X_r})
		\int_0^r K^H(r,u)dW_u}{(s-r)^{\alpha+1}}drdW_s.
		\end{align*}
	Simplify into a double integral form: 
	\begin{align*}
		&\int_0^s \frac{s^{-\alpha} b_x(\phi_s,\mathcal{L}_{X_s})
		\int_0^s K^H(s,u)dW_u
		}{(s-r)^{\alpha+1}}dr-\int_0^s \frac{r^{-\alpha} b_x(\phi_r,\mathcal{L}_{X_r})
		\int_0^r K^H(r,u)dW_u}{(s-r)^{\alpha+1}}dr\\
	=&\int_0^s \int_0^s \frac{s^{-\alpha} b_x(\phi_s,\mathcal{L}_{X_s})K^H(s,u)-r^{-\alpha} b_x(\phi_r,\mathcal{L}_{X_r})
		 K^H(r,u)I_{u<r}}{(s-r)^{\alpha+1}}	dW_udr.
	\end{align*}
	Using \eqref{kh1/2}, we obtain the following estimates:
\begin{align}
	\notag &\int_0^s \left( \int_0^r \left( s^{-\alpha} b_x(\phi_s, \mathcal{L}_{X_s}) K^H(s,u) - s^{-\alpha} b_x(\phi_s, \mathcal{L}_{X_s}) K^H(r,u) \right)^2 du \right)^{\frac{1}{2}} \\
	\notag &\quad \quad\quad \cdot\frac{1}{(s - r)^{\alpha + 1}} dr \\
	\notag &\leq \int_0^s s^{-\alpha} \Vert b_x \Vert_\infty \left( \int_0^r \left( K^H(s,u) - K^H(r,u) \right)^2 du \right)^{\frac{1}{2}} \frac{1}{(s - r)^{\alpha + 1}} dr \\
	\notag &\leq \int_0^s s^{-\alpha} \Vert b_x \Vert_\infty c_H \alpha \left( \int_0^r u^{-2\alpha} \left( \int_u^s (\theta - u)^{\alpha - 1} \theta^\alpha d\theta \right)^2 du \right)^{\frac{1}{2}} \frac{1}{(s - r)^{\alpha + 1}} dr \\
	\notag &\leq \int_0^s \Vert b_x \Vert_\infty c_H \alpha \left( \int_0^r u^{-2\alpha} (r - u)^{2\alpha - 2} du \right)^{\frac{1}{2}} \frac{1}{(s - r)^{\alpha}} dr \\
	\notag &\leq \int_0^s \Vert b_x \Vert_\infty c_H \alpha B(1 - 2\alpha, 2\alpha - 1)^{\frac{1}{2}} r^{-\frac{1}{2}} \frac{1}{(s - r)^{\alpha}} dr \\
	&< \infty, \label{f1}
\end{align}
\begin{align}
	\notag &\int_0^s \left( \int_0^r \left( s^{-\alpha} b_x(\phi_s, \mathcal{L}_{X_s}) K^H(r,u) - r^{-\alpha} b_x(\phi_r, \mathcal{L}_{X_r}) K^H(r,u) \right)^2 du \right)^{\frac{1}{2}} \\
\notag	&\quad\quad\quad \cdot \frac{1}{(s - r)^{\alpha + 1}} dr \\
	\notag &\leq \int_0^s \left| s^{-\alpha} b_x(\phi_s, \mathcal{L}_{X_s}) - r^{-\alpha} b_x(\phi_r, \mathcal{L}_{X_r}) \right| \left( \int_0^r K^H(r,u)^2 du \right)^{\frac{1}{2}} \frac{1}{(s - r)^{\alpha + 1}} dr \\
	\notag &\leq c_H \alpha \int_0^s \left| s^{-\alpha} b_x(\phi_s, \mathcal{L}_{X_s}) - r^{-\alpha} b_x(\phi_r, \mathcal{L}_{X_r}) \right| \\
	\notag &\qquad\quad\quad\cdot \left( \int_0^r u^{-2\alpha} \left( \int_u^s (\theta - u)^{\alpha - 1} \theta^\alpha d\theta \right)^2 du \right)^{\frac{1}{2}} \frac{1}{(s - r)^{\alpha + 1}} dr \\
	\notag &\leq c_H \alpha \int_0^s \left| s^{-\alpha} b_x(\phi_s, \mathcal{L}_{X_s}) - r^{-\alpha} b_x(\phi_r, \mathcal{L}_{X_r}) \right| \\
	\notag &\quad\qquad\quad \cdot \left( \int_0^r u^{-2\alpha} (r - u)^{2\alpha - 1} r^{2\alpha} du \right)^{\frac{1}{2}} \frac{1}{(s - r)^{\alpha + 1}} dr \\
	\notag &\leq c_H \alpha \int_0^s \left| s^{-\alpha} b_x(\phi_s, \mathcal{L}_{X_s}) - r^{-\alpha} b_x(\phi_r, \mathcal{L}_{X_r}) \right| r^\alpha B(1 - 2\alpha, 2\alpha)^{\frac{1}{2}} \frac{1}{(s - r)^{\alpha + 1}} dr \\
	\notag &\leq c_H \alpha B(1 - 2\alpha, 2\alpha)^{\frac{1}{2}} \left( \int_0^s \left| s^{-\alpha} b_x(\phi_s, \mathcal{L}_{X_s}) - r^{-\alpha} b_x(\phi_s, \mathcal{L}_{X_s}) \right| r^\alpha \frac{1}{(s - r)^{\alpha + 1}} dr \right. \\
	\notag &\quad\qquad\qquad\qquad\qquad \left. + \int_0^s \left| r^{-\alpha} b_x(\phi_s, \mathcal{L}_{X_s}) - r^{-\alpha} b_x(\phi_r, \mathcal{L}_{X_r}) \right| r^\alpha \frac{1}{(s - r)^{\alpha + 1}} dr \right) \\
	\notag &\leq c_H \alpha B(1 - 2\alpha, 2\alpha)^{\frac{1}{2}} \left( \Vert b_x \Vert_\infty s^{-\alpha} \int_0^1 (x^\alpha - 1)(1 - x)^{-1 - \alpha} dx \right. \\
	\notag &\quad\qquad\qquad\qquad\qquad\quad + \left. \int_0^s \Vert b_x(\phi_s, \mathcal{L}_{X_s}) \Vert_H \frac{1}{(s - r)^{\frac{1}{2}}} \right) \\
	&< \infty, \label{f2}
\end{align}
and
\begin{align}
	\notag &\int_0^s \left( \int_r^s s^{-2\alpha} b_x(\phi_s, \mathcal{L}_{X_s})^2 K^H(s,u)^2 du \right)^{\frac{1}{2}} \frac{1}{(s - r)^{\alpha + 1}} dr \\
	\notag &\leq \int_0^s s^{-\alpha} \Vert b_x \Vert_\infty \left( \int_r^s K^H(s,u)^2 du \right)^{\frac{1}{2}} \frac{1}{(s - r)^{\alpha + 1}} dr \\
	\notag &\leq \int_0^s \Vert b_x \Vert_\infty c_H \left( \int_r^s u^{-2\alpha} (s - u)^{2\alpha} du \right)^{\frac{1}{2}} \frac{1}{(s - r)^{\alpha + 1}} dr \\
	\notag &\leq \int_0^s \Vert b_x \Vert_\infty c_H r^{-\alpha} (s - r)^{\alpha + 1} \frac{1}{(s - r)^{\alpha + 1}} dr \\
	&< \infty. \label{f3}
\end{align}

	Therefore, from \eqref{f1}, \eqref{f2}, 
	\eqref{f3} and Theorem \ref{fubini}, we obtain that
	\begin{equation*}
	 	C_2=\int_0^1\int_0^s f(s,u)dW_udW_s,
	 \end{equation*}
	 where
	 \begin{align}
	 \notag	f(s,u)=&\frac{1}{\Gamma(1-\alpha)}
	 	(
	 		s^{-\alpha}b_x(\phi_s,\mathcal{L}_{X_s})K^H(s,u)
	 		+\int_0^s \frac{\alpha b_x(\phi_s,\mathcal{L}_{X_s})K^H(s,u)
		}{(s-r)^{\alpha+1}}dr\\
		& +\int_u^s \alpha s^\alpha \frac{r^{-\alpha} b_x(\phi_r,\mathcal{L}_{X_r})K^H(r,u)}{(s-r)^{\alpha+1}}dr).\label{f2'}
	 \end{align}
	 Set    
	 \begin{equation*}
	 	\tilde{f}(s,u)=\frac{f(s,u)I_{u\leq s}+f(u,s)I_{s<u}}{2},
	 \end{equation*}
	where $\tilde{f}(s,u)$ is a continuous and symmetric function on $[0,1]\times [0,1]$ and  
	 \begin{align*}
	 	\int_0^1\int_0^1 \tilde{f}(s,u) dW_u dW_s&=\int_0^1\int_0^s \frac{\tilde{f}(s,u)}{2} dW_u dW_s
	 	+\int_0^1\int_s^1 \frac{\tilde{f}(u,s)}{2} dW_udW_s \\
	 	&=\int_0^1\int_0^s \frac{\tilde{f}(s,u)}{2} dW_u dW_s
		+\int_0^1\int_0^u \frac{\tilde{f}(u,s)}{2} dW_sdW_u\\
		&=C_2.
	 \end{align*}
	Proofing that $K(\tilde{f})$ is nuclear can be referenced in \cite[Lemma 14]{Nualart}.
	 According to Theorem \ref{thm3}, the exponential condition limit of $C_2$ can be obtained from the trace of $\tilde{f}(s,u)$ :
	 \begin{align}
	 	\notag	\lim_{\varepsilon\to 0} \mathbb{E}(\exp(cC_2)\mid\Vert B^H\Vert_\beta \leq \varepsilon)&=\exp\left(-c\operatorname{Tr} \tilde{f}\right)=\exp\left(-\int_0^1 c\tilde{f}(s,s)ds\right)\\&=\exp\left(-\int_0^1 \frac{c}{2} f(s,s)ds\right) .\label{C_2 '}
	 \end{align}
	We now compute the limit of $f(s,u)$ as $u $ approaches $s$.
	Using \eqref{kh1/2}, the first term on the right-hand side of \eqref{f2'} vanishes when $u$ approaches $s$. For the rest term we have 
		\begin{align*}
		&	\frac{1}{\Gamma(1-\alpha)}\left(\int_0^s \frac{\alpha b_x(\phi_s,\mathcal{L}_{X_s})K^H(s,u)
		}{(s-r)^{\alpha+1}}dr
		 +\int_u^s \alpha s^\alpha \frac{r^{-\alpha} b_x(\phi_r,\mathcal{L}_{X_r})K^H(r,u)}{(s-r)^{\alpha+1}}dr\right)\\
		=& \frac{\alpha b_x(\phi_s,\mathcal{L}_{X_s})K^H(s,u)}{\Gamma(1-\alpha)}
		\int_0^u \frac{1
		}{(s-r)^{\alpha+1}}dr \\
		&+\frac{\alpha s^\alpha}{\Gamma(1-\alpha)}
		\int_u^s \frac{s^{-\alpha}b_x(\phi_s,\mathcal{L}_{X_s})K^H(s,u)
		-r^{-\alpha}b_x(\phi_r,\mathcal{L}_{X_r})K^H(r,u)
		}{(s-r)^{\alpha+1}}dr \\
		=&D_1(s,u)+D_2(s,u),
		\end{align*}
	where
	\begin{equation*}
		D_1(s,u)=\frac{\alpha b_x(\phi_s,\mathcal{L}_{X_s})K^H(s,u)}{\Gamma(1-\alpha)}
		\int_0^u \frac{1
		}{(s-r)^{\alpha+1}}dr,
	\end{equation*}
	\begin{equation*}
		D_2(s,u)=\frac{\alpha s^\alpha}{\Gamma(1-\alpha)}
		\int_u^s \frac{s^{-\alpha}b_x(\phi_s,\mathcal{L}_{X_s})K^H(s,u)
		-r^{-\alpha}b_x(\phi_r,\mathcal{L}_{X_r})K^H(r,u)
		}{(s-r)^{\alpha+1}}dr.
	\end{equation*}
	From \eqref{kh1/2}, we have
		\begin{align*}
			D_1(s,u)=&\frac{\alpha b_x(\phi_s,\mathcal{L}_{X_s})K^H(s,u)}{\Gamma(1-\alpha)}
		\int_0^u \frac{1
		}{(s-r)^{\alpha+1}}dr \\
		=& \frac{\alpha c_H b_x(\phi_s,\mathcal{L}_{X_s})}{\Gamma(1-\alpha)}u^{-\alpha}
		((s-u)^{-\alpha}-s^{-\alpha})(s-u)^\alpha
		\int_0^1 v^{\alpha-1}((s-u)v+u)^\alpha dv .
		\end{align*}
	Thus, letting \( u \to s \), we obtain that 
\[
D_1(s,s) = \frac{c_H}{\Gamma(1-\alpha)} b_x(\phi_s, \mathcal{L}_{X_s}).
\]

Next, we divide $D_2$ into three parts:  
\begin{align*}
    D_2(s,u) &= \frac{\alpha s^\alpha}{\Gamma(1-\alpha)}
    \int_u^s \frac{s^{-\alpha} b_x(\phi_s, \mathcal{L}_{X_s}) K^H(s,u)
    - r^{-\alpha} b_x(\phi_r, \mathcal{L}_{X_r}) K^H(r,u)}{(s-r)^{\alpha+1}} \, dr \\
    &= \frac{\alpha s^\alpha}{\Gamma(1-\alpha)} \big(E_1(s,u) + E_2(s,u) + E_3(s,u)\big),
\end{align*}
where
\begin{align*}
	E_1(s,u) &=K^H(s,u) b_x(\phi_s,\mathcal{L}_{X_s})
		\int_u^s \frac{s^{-\alpha}-r^{-\alpha}}{(s-r)^{\alpha+1}}dr,\\
	E_2(s,u) &=K^H(s,u) \int_u^s \frac{b_x(\phi_s,\mathcal{L}_{X_s})
		-b_x(\phi_r,\mathcal{L}_{X_r})
}{(s-r)^{\alpha+1}} r^{-\alpha}dr,\\
	E_3(s,u) &=\int_u^s \frac{K^H(s,u)-K^H(r,u)}{(s-r)^{\alpha+1}}r^{-\alpha}b_x(\phi_r,\mathcal{L}_{X_r})dr.
\end{align*}
	Since 
	\begin{equation*}
		\left|\int_u^s \frac{s^{-\alpha}-r^{-\alpha}}{(s-r)^{\alpha+1}}dr\right|\leq 
		\alpha u^{-\alpha-1}\int_u^s 
		\frac{1}{(s-r)^{\alpha}}dr\leq \frac{\alpha}{1-\alpha}u^{-\alpha-1}(s-u)^{1-\alpha},
	\end{equation*}
	we have 
	\begin{equation*}
		|E_1(s,u)|\leq |K^H(s,u) b_x(\phi_s,\mathcal{L}_{X_s})|\frac{\alpha}{1-\alpha}u^{-\alpha-1}(s-u)^{1-\alpha},
	\end{equation*}
	hence $E_1(s,s)=0.$
	
	Since \(\phi\) is \(H\)-Hölder continuous and \(b_x\) is Lipschitz continuous with constant \(L_2\), we have:  
\begin{align*}
    |E_2(s,u)| &= \left| K^H(s,u) \int_u^s \frac{b_x(\phi_s, \mathcal{L}_{X_s}) - b_x(\phi_r, \mathcal{L}_{X_r})}{(s-r)^{\alpha+1}} r^{-\alpha} \, dr \right| \\
    &\leq \left| K^H(s,u) \int_u^s \frac{L_2 \big( |\phi_s - \phi_r| + 2\Vert b\Vert_\infty|s-r|+2|s-r|^H \big)}{(s-r)^{\alpha+1}} r^{-\alpha} \, dr \right| \\
    &\leq L_2 \big(  \|\phi\|_H +2+2\Vert b\Vert_\infty \big) |K^H(s,u)| u^{-\alpha} \int_u^s \frac{1}{(s-r)^{\frac{1}{2}}} \, dr \\
    &\leq 2 L_2 \big(  \|\phi\|_H +2+2\Vert b\Vert_\infty \big) |K^H(s,u)| u^{-\alpha} (s-u)^{\frac{1}{2}},
\end{align*}
which implies that \(E_2(s, s) = 0\).

	Using the expression $\eqref{kh1/2}$ of the kernel $K^H(t,s)$, we can write  
\begin{align*}
    E_3(s,u) &= \int_u^s \frac{K^H(s,u)-K^H(r,u)}{(s-r)^{\alpha+1}} r^{-\alpha} b_x(\phi_r,\mathcal{L}_{X_r}) \, dr \\
    &= c_H \alpha u^{-\alpha} \int_u^s \frac{\int_r^s (\theta-u)^{\alpha-1} \theta^\alpha \, d\theta}{(s-r)^{\alpha+1}} r^{-\alpha} b_x(\phi_r,\mathcal{L}_{X_r}) \, dr.
\end{align*}

Set \(v = \frac{r-u}{s-u}\) and \(z = \frac{\theta-u}{s-u}\), which yields:  
\begin{align*}
    E_3(s,u) &= c_H \alpha u^{-\alpha} \int_0^1 \frac{\int_{v(s-u)+u}^s (\theta-u)^{\alpha-1} \theta^\alpha \, d\theta}{(s-u-v(s-u))^{\alpha+1}} (v(s-u)+u)^{-\alpha} \\
    &\qquad\quad\qquad\quad \cdot b_x(\phi_{v(s-u)+u}, \mathcal{L}_{X_{v(s-u)+u}})(s-u) \, dv \\
    &= c_H \alpha u^{-\alpha} \int_0^1 \frac{\int_{v(s-u)+u}^s (\theta-u)^{\alpha-1} \theta^\alpha \, d\theta}{(1-v)^{\alpha+1}(s-u)^\alpha} (v(s-u)+u)^{-\alpha} \\
    &\quad\quad\quad\qquad\quad  \cdot b_x(\phi_{v(s-u)+u}, \mathcal{L}_{X_{v(s-u)+u}}) \, dv \\
    &= c_H \alpha u^{-\alpha} \int_0^1 \frac{b_x(\phi_{v(s-u)+u}, \mathcal{L}_{X_{v(s-u)+u}})}{(1-v)^{\alpha+1}(s-u)^\alpha} (v(s-u)+u)^{-\alpha} \\
    &\quad\quad \quad \qquad\quad \cdot \int_v^1 z^{\alpha-1}(z(s-u)+u)^\alpha \, dz \, dv \\
    &= c_H \alpha u^{-\alpha} \int_0^1 \frac{b_x(\phi_{v(s-u)+u}, \mathcal{L}_{X_{v(s-u)+u}})}{(1-v)^{\alpha+1}} (v(s-u)+u)^{-\alpha} \\
    &\qquad\quad\qquad\quad\cdot \int_v^1 z^{\alpha-1}(z(s-u)+u)^\alpha \, dz \, dv.
\end{align*}
	When $s=u$ we have 
	\begin{align*}
		E_3(s,s)=& c_H \alpha s^{-\alpha}
			\int_0^1 \frac{b_x(\phi_{s},\mathcal{L}_{X_{s}})}{(1-v)^{\alpha+1}}s^{-\alpha}
			 \int_{v}^1 z^{\alpha-1}s^\alpha dz dv \\
			 =& c_H s^{-\alpha} b_x(\phi_{s},\mathcal{L}_{X_{s}}) 
			 \int_0^1 (1-v)^{-\alpha-1}(1-v^\alpha)dv \\
			 =& c_H s^{-\alpha} b_x(\phi_{s},\mathcal{L}_{X_{s}}) 
			\left( (1-v)^{-\alpha}(1-v^\alpha) \Big|_0^1  +\int_0^1 v^{\alpha-1}(1-v^\alpha)dv  \right)\\
			 =&c_H s^{-\alpha} b_x(\phi_{s},\mathcal{L}_{X_{s}})\left(\frac{\Gamma(\alpha)\Gamma(1-\alpha)}{\Gamma(1)} -\frac{1}{\alpha}\right).
	\end{align*}
	Finally,
		\begin{align*}
			f(s,s)&=D_1(s,s)+D_2(s,s)=D_1(s,s)+\frac{\alpha s^\alpha}{\Gamma(1-\alpha)} E_3(s,s)\\
			&= \frac{c_H}{\Gamma(1-\alpha)}b_x(\phi_{s},\mathcal{L}_{X_{s}})
			+\frac{\alpha s^\alpha}{\Gamma(1-\alpha)}c_H s^{-\alpha} b_x(\phi_{s},\mathcal{L}_{X_{s}})\left(\frac{\Gamma(\alpha)\Gamma(1-\alpha)}{\Gamma(1)} -\frac{1}{\alpha}\right)\\
			&=\alpha c_H b_x(\phi_{s},\mathcal{L}_{X_{s}})\Gamma(\alpha)\\
			&=c_H b_x(\phi_{s},\mathcal{L}_{X_{s}})\Gamma(\alpha+1).
		\end{align*}
According to \eqref{C_2 '}, we obtain that 
\begin{align*}
	\lim_{\varepsilon \to 0} \mathbb{E}\left(\exp(cC_2) \mid \Vert B^H\Vert_\beta \leq \varepsilon \right) 
&= \exp\left(-\int_0^1 \frac{c}{2} f(s,s) \, ds\right)\\
&= \exp\left(-\frac{cd_H}{2}\int_0^1 b_x(\phi_{s},\mathcal{L}_{X_{s}}) \, ds\right).
\end{align*}

It only remains to study the behaviour of the term \(C_3\). For any \(c \in \mathbb{R}\), we define:  
\[
M_t = c \int_0^t s^\alpha D^\alpha_{0^+} s^{-\alpha} R_s \, dW_s.
\]
	Then $M_t$ is a martingale. 
	In order to estimate the quadratic variation of $M_t$, we will represent the remainder term in the following form:
	\begin{align*}
		R_s&= b(B^H_s+\phi_{s},\mathcal{L}_{X_{s}})-b(\phi_{s},\mathcal{L}_{X_{s}})-b_x(\phi_{s},\mathcal{L}_{X_{s}})B^H_s\\
		&= \int_0^1 b_x(\lambda B^H_s+ \phi_{s},\mathcal{L}_{X_{s}})-b_x(\phi_{s},\mathcal{L}_{X_{s}})B^H_s d\lambda\\
		&= \int_0^1 \int_0^\lambda 
		b_{xx}(\mu B^H_s +\phi_{s},\mathcal{L}_{X_{s}})(B^H_s)^2 d\mu d\lambda.
	\end{align*}
	Then using the fact that \( b_{xx} \) is Lipschitz continuous with constant \( L_3 \) and bounded, we have
\begin{align*}
    |R_s - R_r| 
    &\leq (B^H_s)^2 \int_0^1 \int_0^\lambda 
    |b_{xx}(\mu B^H_s + \phi_{s}, \mathcal{L}_{X_{s}})
    - b_{xx}(\mu B^H_r + \phi_{r}, \mathcal{L}_{X_{r}})| \, d\mu d\lambda \\
    &\quad + |(B^H_s)^2 - (B^H_r)^2| 
    \int_0^1 \int_0^\lambda 
    |b_{xx}(\mu B^H_s + \phi_{s}, \mathcal{L}_{X_{s}})| \, d\mu d\lambda \\
    &\leq \int_0^1 \int_0^\lambda L_3 (B^H_s)^2 (\mu |B^H_s - B^H_r| + |\phi_s - \phi_r| + 2\Vert b\Vert_\infty |s-r|+2|s-r|^H ) \, d\mu d\lambda \\
    &\quad + \int_0^1 \int_0^\lambda \mu \Vert b_{xx} \Vert_\infty |(B^H_s)^2 - (B^H_r)^2| \, d\mu d\lambda \\
    &\leq L_3 (|\phi_s - \phi_r| + |B^H_s - B^H_r| + 2\Vert b\Vert_\infty |s-r|+2|s-r|^H)(B^H_s)^2 \\
    &\quad + \Vert b_{xx} \Vert_\infty |(B^H_s)^2 - (B^H_r)^2|.
\end{align*}
By combining this with \eqref{Weil}, we obtain that there exists a constant $\bar{c}$ such that
\begin{align*}
    &\Gamma(1-\alpha) |s^\alpha D^\alpha_{0^+} s^{-\alpha} R_s| \\
    &= \left|s^{-\alpha} R_s + \alpha s^\alpha \int_0^s \frac{s^{-\alpha} R_s - r^{-\alpha} R_r}{(s-r)^{\alpha+1}} \, dr\right| \\
    &\leq k s^{-\alpha} \varepsilon^2
    + \left|\alpha s^\alpha \int_0^s \frac{s^{-\alpha} R_s - s^{-\alpha} R_r + s^{-\alpha} R_r - r^{-\alpha} R_r}{(s-r)^{\alpha+1}} \, dr\right| \\
    &\leq k s^{-\alpha} \varepsilon^2 
    + \alpha \left|\int_0^s \frac{R_s - R_r}{(s-r)^{\alpha+1}} \, dr\right|
    + a k s^\alpha \varepsilon^2 \left|\int_0^s \frac{s^{-\alpha} - r^{-\alpha}}{(s-r)^{\alpha+1}} \, dr\right| \\
    &\leq k s^{-\alpha} \varepsilon^2 + c_\alpha s^\alpha \varepsilon^2 \\
    &+ \alpha \int_0^s \frac{L_3 (|\phi_s - \phi_r| + |B^H_s - B^H_r| + 2\Vert b\Vert_\infty |s-r|+2|s-r|^H)(B^H_s)^2
     }{(s-r)^{\alpha+1}}   \\
    &\qquad\quad+\frac{\Vert b_{xx} \Vert_\infty |(B^H_s)^2 - (B^H_r)^2|}{(s-r)^{\alpha+1}}dr\\
    &\leq k s^{-\alpha} \varepsilon^2 + c_\alpha s^\alpha \varepsilon^2 \\
    &\quad + \alpha \int_0^s 
    \frac{L_3 (\Vert \phi \Vert_H (s - r)^H + \Vert B^H \Vert_\beta (s - r)^\beta +2\Vert b\Vert_\infty |s-r|+2|s-r|^H)(B^H_s)^2}{(s-r)^{\alpha+1}} \\
    &\quad \qquad\quad+ \frac{\Vert b_{xx} \Vert_\infty |B^H_s + B^H_r| \Vert B^H \Vert_\beta (s-r)^\beta}{(s-r)^{\alpha+1}} \, dr \\
    &\leq \bar{c} \varepsilon^2.
\end{align*}
Therefore, the quadratic variation of \(M_t\) satisfies
\[
\langle M \rangle_t = c^2 \int_0^t (s^\alpha D_{0^+}^\alpha s^{-\alpha} R_s)^2 \, ds \leq k' \varepsilon^4,
\]where $k'=\bar{c}c^2$.
Applying the  exponential inequality for martingales we have
 		\begin{equation*}
 			P\left(\left|\int_0^1 s^{\alpha}D_{0^+}^\alpha s^{-\alpha} R_sdW_s\right|> \xi, \Vert B^H\Vert _\beta \leq \varepsilon\right)
 	\leq \exp\left(-\frac{\xi^2}{2k'\varepsilon^4}\right).
	\end{equation*}
 	Combining with the small ball behaviour of the fractional Brownian motion under the Hölder norm \eqref{sbb holder}, we get 
 	\begin{equation}
 	P\left(\left|\int_0^1 s^{\alpha}D_{0^+}^\alpha s^{-\alpha}R_sdW_s\right|>\xi\mid \Vert B^H\Vert _\beta \leq \varepsilon\right)
 	\leq \exp\left(-\frac{\xi^2}{2k'\varepsilon^4}\right)\exp(K_2\varepsilon^\frac{-1} {H-\beta}) .\label{p2}
 \end{equation}

 Under the condition $\Vert B^H\Vert \leq \varepsilon$, let $p_{\varepsilon,t}$ be the measure on $\mathbb{R}$ induced by $M_t$, and let $F_{\varepsilon,t}$ be the distribution function of $M_t$. 	
	Hence, for any $\delta>0,$ by \eqref{p2} we have 
		\begin{align*}
			&\mathbb{E}(\exp(cC_3)\mid\Vert B^H\Vert_\beta\leq\varepsilon)\\
			=&\mathbb{E}(\exp(M_t)\mid\Vert B^H\Vert_\beta\leq\varepsilon)\\
 		=&\int_{-\infty}^\infty e^\xi p_{\varepsilon,t}(d\xi)\\
 		=&\int_{-\infty}^{\delta} e^\xi p_{\varepsilon,t}(d\xi) +\int_{\delta}^{+\infty} e^\xi  p_{\varepsilon,t}(d\xi)        \\
 		\leq & e^{\delta}- e^\xi(1-F(\xi))\Big|^\infty_\delta 
 		+\int_{\delta}^{+\infty} e^\xi (1-F(\xi)) d\xi \\
 		  \leq & e^{\delta}+\int_{\delta}^{+\infty} e^\xi P\left(\left|c\int_0^1s^{-\alpha}I_{0+}^\alpha s^\alpha R_sdW_s\right|>\xi\mid\Vert B^H\Vert_\beta \leq \varepsilon\right) d\xi\\
 		 &+e^{\delta}P\left(\left|c\int_0^1s^{-\alpha}I_{0+}^\alpha s^\alpha R_sdW_s\right|>\delta \mid\Vert B^H\Vert_\beta \leq \varepsilon\right)\\
 		 \leq & e^\delta +e^\delta   \exp\left(-\frac{\delta^2}{2k'\varepsilon^4}\right)\exp(K_2\varepsilon^\frac{-1} {H-\beta})+ \int_{\delta}^{+\infty} e^\xi\exp\left(-\frac{\xi^2}{2k'\varepsilon^4}\right)\exp(K_2\varepsilon^\frac{-1} {H-\beta})d\xi \\
			\leq & e^\delta +e^\delta   \exp\left(-\frac{\delta^2}{2k'\varepsilon^4}\right)\exp(K_2\varepsilon^\frac{-1} {H-\beta})+\exp\left(\frac{k'\varepsilon^4}{2}+K_2\varepsilon^{-\frac{1}{H-\beta}}\right) \int_{\frac{\delta}{\sqrt{2k'}\varepsilon^2}-\frac{\sqrt{2k'}\varepsilon^2}{2}}^{+\infty} e^{-x^2}dx\\
			\leq & e^\delta +e^\delta   \exp\left(-\frac{\delta^2}{2k'\varepsilon^4}\right)\exp(K_2\varepsilon^\frac{-1} {H-\beta})\\
			&+\sqrt{2k'} \frac{\pi}{8}\exp\left(\frac{k'\varepsilon^4}{2}+K_2\varepsilon^{-\frac{1}{H-\beta}}-\left({\frac{\delta}{\sqrt{2k'}\varepsilon^2}-\frac{\sqrt{2k'}\varepsilon^2}{2}}\right)^2\right)\varepsilon^2,
		\end{align*}
	 when $H-\frac{1}{2}<\beta< H-\frac{1}{4}$.
	Thus, when $H-\frac{1}{2}<\beta< H-\frac{1}{4}$, we have
		\begin{align*}
			\lim_{\varepsilon\to 0} \mathbb{E}(\exp(cC_3)\mid\Vert B^H\Vert_\beta\leq\varepsilon)  \leq e^\delta.
		\end{align*}
	By the arbitrariness of $\delta$, it can be concluded that
	\begin{equation*}
		\lim_{\varepsilon\to 0} \mathbb{E}(\exp(cC_3)\mid\Vert B^H\Vert_\beta\leq\varepsilon)\leq 1, \quad \forall c\in \mathbb{R}.
	\end{equation*}
	Hence
	\begin{equation}
		\lim_{\varepsilon\to 0} \mathbb{E}(\exp(A_1)\mid\Vert B^H\Vert_\beta\leq\varepsilon)=\lim_{\varepsilon\to 0} \mathbb{E}(\exp(C_1+C_2+C_3)\mid\Vert B^H\Vert_\beta\leq\varepsilon)=1.\label{ar1}
	\end{equation}
	
	In summary, based on \eqref{ar1}, \eqref{ar2}, \eqref{ar3}, and \eqref{ar4}, and by further applying Theorem \ref{lem8}, we can derive our result:
\begin{align*}
&\lim_{\varepsilon\to 0}\frac{P(\Vert X - \phi \Vert_\beta \leq \varepsilon)}{P(\Vert B^H \Vert_\beta \leq \varepsilon)}\\ 
=& \lim_{\varepsilon\to 0}\mathbb{E}\Big( \exp(A_1 + A_2 + A_3 + A_4) I_{\Vert B^H \Vert_{\beta} \leq \varepsilon} \Big) \\
&\quad \cdot \exp\left( -\frac{1}{2} \int_0^1 \big( \dot{\phi}_s - s^{\alpha} D_{0^+}^\alpha s^{-\alpha} b(\phi_s, \mathcal{L}_{X_s}) \big)^2 ds\right)  \\
=& \exp\left( -\frac{1}{2} \int_0^1 \left( \dot{\phi}_s - s^{\alpha} D_{0^+}^\alpha s^{-\alpha} b(\phi_s, \mathcal{L}_{X_s}) \right)^2 + d_H b_x(\phi_s, \mathcal{L}_{X_s})  ds \right) \\
=& \exp\left( J(\phi, \dot{\phi}) \right).
\end{align*}
Finally, we obtain 
\begin{equation*}
	J(\phi,\dot{\phi})=-\frac{1}{2} \int_0^1 (\dot{\phi}_s-s^{\alpha}D_{0^+}^\alpha s^{-\alpha}b(\phi_s,\mathcal{L}_{X_s} ))^2+d_H b_x(\phi_s,\mathcal{L}_{X_s})ds,
\end{equation*}
where 
	\begin{equation*}
   d_H=\left(\frac{2H\Gamma(\frac{3}{2}-H)\Gamma(H+\frac{1}{2})}{\Gamma(2-2H)}\right)^\frac{1}{2}.
	\end{equation*}

	\end{proof}
	
	\section{The finite-dimensional case of the main result}
	 In this section, we will present the finite-dimensional case of the main result in Section 3.
		 Consider the following distribution-dependent stochastic differential equation in $\mathbb{R}^n$
\begin{equation}\label{ndsde}
    X_t = x + \int_0^t b(X_s, \mathcal{L}_{X_s}) \, ds + B^H_t,
\end{equation}
where $\mathcal{L}_{X_s}$ denotes the law of $X_s$,  $B^H_t = (B^{H}_{1,t},\dots,B^{H}_{n,t})^T$ is an $n$-dimentional fractional Brownian motion with Hurst parameter $H \in (0,1)$ and $B^H_{i,t}$ are mutually independent. 
	
	For a vector-valued function \( f = (f_1, \dots, f_n) \), its norm can be extended from the one-dimensional case as follows  
\begin{equation*}  
\Vert f\Vert = \left( \sum_{i=1}^{n} \Vert f_i\Vert^2 \right)^{\frac{1}{2}}.  
\end{equation*}

	Similarly, we also need to provide the assumptions on the drift term in the finite-dimensional case.
	\begin{itemize}
    \item (\textbf{Hyp 1}) 
    \begin{enumerate}
    	\item The Jacobian matrix $\nabla_x b$ with respect to the spatial variables exists and is continuous;
    	\item  \( b \) and \( \nabla_x b \) are  bounded;
		 \item $b$ is Lipschitz continuous with constant $L_1$;
		 \item There exists a constant $k$ such  
		that 
		\begin{equation*}
			|b(x+h,\mu)-b(x,\mu)-\nabla_x b(x,\mu)h|\leq k |h|^2,\quad \forall x,h\in \mathbb{R}^n, \mu \in \mathscr{P}_2(\mathbb{R}^n).
		\end{equation*}
		
                   \end{enumerate}
\end{itemize}

 \begin{itemize}
    \item (\textbf{Hyp 2}) 
    \begin{enumerate}
    	\item The Jacobian matrix $\nabla_x b$ and the Hessian Tensor $\nabla_x^2 b$ with respect to the spatial variables exist;
    	\item  \( b \), \( \nabla_x b \), and \( \nabla_x^2 b \) are all bounded;
		\item $b,\nabla_x b$ and $\nabla_x^2 b$ are Lipschitz continuous with constants 
        $L_1$, $L_2$ and $L_3$.
        \end{enumerate}
        \end{itemize}
	To obtain the result in the finite-dimensional case, the following two lemmas are also required.
	\begin{lemma}[Theorem 5.1, \cite{karatzas1991brownian}]
		Assume that \label{ngir}
		\begin{equation*}
			Z_t(X)\triangleq \exp\left[\sum_{i=1}^n \int_0^t X_s^{(i)}dW_s^{i}-\frac{1}{2}\int_0^1 \Vert X_s\Vert^2ds \right]
		\end{equation*}
		is a martingale. Define a process $\tilde{W}=\{\tilde{W_t}=(\tilde{W_t}^{(1)},\dots,\tilde{W_t}^{(n)}),\mathcal{F}_t;0\leq t<\infty \}$ by
		\begin{equation*}
			\tilde{W_t}^{(i)}\triangleq W_t^{(i)}-\int_0^t X_s^{i}ds;\quad 1\leq i\leq n,\quad 0\leq t<\infty.
		\end{equation*}
		For each fixed $T\in [0,\infty)$, the process $\{\tilde{W_t},\mathcal{F}_t;0\leq t\leq T \}$ is a n-dimensional Brownian motion on $(\Omega,\mathcal{F}_t,\tilde{P_T}),$ where
		\begin{equation*}
			\tilde{P_T}(A)\triangleq \mathbb{E}[I_{A}Z_T(X)];\quad A\in \mathcal{F}_T.
		\end{equation*}
	\end{lemma}
	
	\begin{lemma}[Lemma 3.6, \cite{maayan2017onsagermachlupfunctionalassociatedadditive}]\label{nex}
		Let \( F_i \) be the \( \sigma \)-field generated by \( \{ B^H_{1,t}, \dots, B^H_{i-1,t}, B^H_{i+1,t}, \dots, B^H_{n,t};\\ 0 \leq t \leq 1 \} \). Let \( \Psi(\cdot) \) be an \( F_i \)-adapted function such that
\[
\lim_{\varepsilon \to 0} \mathbb{E}\left( \exp\left( c \int_0^1 \big((K^H)^{-1}(\Psi)\big)^2(s) \, ds \right) \middle| \|B^H\| < \varepsilon \right) = 1,\quad \forall c \in \mathbb{R}_+.
\]
Then
\begin{equation*}
	\lim_{\varepsilon \to 0} \mathbb{E}\left( \exp\left(  \int_0^1 (K^H)^{-1} \Psi(s) \, dW_i^t  \right) \middle| \|B^H\| < \varepsilon \right) = 1.
\end{equation*}

	\end{lemma}

	The following theorems constitute the main results in $n$-dimensional case.
	
	\begin{theorem}\label{nsincase}
		Let $X$ be the solution of  equation \eqref{ndsde}. Suppose that $\phi_i$ are functions such that $\phi_i-x_i \in \mathcal{H}^p $ with $p>\frac{1}{H}$, and let $b$ satisfy assumption  \textnormal{(\textbf{Hyp 1})} and $\frac{1}{4}< H<\frac{1}{2}.$ Then the Onsager-Machlup functional of $X$ for 
		norms $\Vert \cdot\Vert _\beta$ with $0<\beta<H-\frac{1}{4}$ and $\Vert \cdot\Vert _\infty$
		can be expressed as follows:
		\begin{equation}
			J(\phi,\dot{\phi})=-\frac{1}{2}
			\int_0^1  |\dot{\phi}_s-s^{-\alpha} I_{0^+}^\alpha s^\alpha b(\phi_s,\mathcal{L}_{X_s} )|^2+d_H \nabla_x \cdot  b(\phi_s,\mathcal{L}_{X_s})ds,\label{noms}
 		\end{equation}
	where 
	\begin{equation*}
		d_H=\left(\frac{2H\Gamma(\frac{3}{2}-H)\Gamma(H+\frac{1}{2})}{\Gamma(2-2H)}\right)^\frac{1}{2},
	\end{equation*}
	$\alpha=\frac{1}{2}-H$ and $K^H\dot{\phi}=\phi-x.$
	\end{theorem}

	\begin{theorem}\label{nrecase}
		Let $X$ be the solution of  equation \eqref{ndsde}. Suppose that $\phi_i$ are  functions such that $\phi_i-x_i \in \mathcal{H}^2 $, and let $b$ satisfy assumption \textnormal{(\textbf{Hyp 2})} and $H>\frac{1}{2}.$ Then the Onsager-Machlup functional of $X$ for the 
		norm $\Vert \cdot\Vert _\beta$ with $H-\frac{1}{2}<\beta<H-\frac{1}{4}$ 
		can be expressed as follows:
		\begin{equation}
			J(\phi,\dot{\phi})=-\frac{1}{2} \int_0^1 |\dot{\phi}_s-s^{\alpha}D_{0^+}^\alpha s^{-\alpha}b(\phi_s,\mathcal{L}_{X_s} )|^2+d_H \nabla_x \cdot b(\phi_s,\mathcal{L}_{X_s})ds,\label{nomr}
		\end{equation}    
	where 
	\begin{equation*}
		d_H=\left(\frac{2H\Gamma(\frac{3}{2}-H)\Gamma(H+\frac{1}{2})}{\Gamma(2-2H)}\right)^\frac{1}{2},
	\end{equation*}
	$\alpha=H-\frac{1}{2}$ and $K^H\dot{\phi}=\phi-x.$
	\end{theorem}
	\begin{proof}
		Our proof method is the same as in the one-dimensional case: We first simplify the ratio using the Girsanov transformation and then compute the conditional expectation. 
		According to  Lemma \ref{ngir},  we set $Y_t=B^H_t+\phi_t\ 
 $
 	and
 		\begin{equation*}
 			\tilde{W}_t=W_t-\int_0^t (K^H)^{-1}\left(\int_0^\cdot b(B^H_r+\phi_r,\mathcal{L}^P_{X_r} )dr\right)(s)-\dot{\phi}_sds,
 		\end{equation*}
 		where $\int_0^s b(B^H_r + \phi_r, \mathcal{L}^P_{X_r})  dr \in I_{0^+}^{H + \frac{1}{2}}(L^2([0,1],\mathbb{R}^n))$.
 		The ratio is simplified using two weak solutions $X$ and $Y$ of the same equation. The difference in the ratio computation lies in the calculation of the $A_4$ term. Using the Taylor expansion, the following type of term appears
\begin{equation*}
	\int_0^1 (K^H)^{-1}\left(\int_0^\cdot b^i_{x_j}(\phi_r,\mathcal{L}_{X_r} ) B^H_{j,r}dr\right)(s) dW_{i,s}.
\end{equation*}
From Lemma \ref{nex}, we know that we only need to compute the exponential conditional expectation corresponding to $i=j$, which is done in the same way as in the one-dimensional case. In conclusion, we obtain the Onsager-Machlup functional in the finite-dimensional setting.
	\end{proof}

	\section{Numerical experiments}
	In this section, we validate our main result through numerical simulations of two specific SDEs.  
Specifically, for a given SDE, we obtain its most probable transition path from state \( X_0 = x_0 \) to state \( X_1 = x_1 \) using the Onsager-Machlup functional.  
We then compare the most probable path with the actual paths of the SDE transitioning from \( x_0 \) to \( x_1 \).

	We can obtain the most probable path using the Euler-Lagrange equation.
	
When $H=\frac{1}{2}$, the Onsager-Machlup functional takes the form:
\begin{equation*}
    J(\phi,\dot{\phi})=-\frac{1}{2} \int_0^1 \left({\phi}_s'-b(\phi_s,\mathcal{L}_{X_s})\right)^2 + b_x(\phi_s,\mathcal{L}_{X_s})\,ds,
\end{equation*}
and the most probable path $\phi^*$ transitioning from state $x_0$ to $x_1$ must satisfy the Euler-Lagrange equation
\begin{equation}
    \frac{\partial L(s,(\phi^*)',\phi^*)}{\partial \phi} = \frac{d}{dt} \frac{\partial L(s,(\phi^*)',\phi^*)}{\partial \phi'} \label{el}
\end{equation}
with boundary conditions $\phi^*_0=x_0$, $\phi^*_1=x_1$, where $L(s,\phi',\phi)=\left({\phi}_s'-b(\phi_s,\mathcal{L}_{X_s})\right)^2 + b_x(\phi_s,\mathcal{L}_{X_s})$. 

For $H \in (\frac{1}{4}, \frac{1}{2}) \cup (\frac{1}{2}, 1)$, the governing equation for the most probable transition path can similarly be derived using variational calculus.

For $\frac{1}{4}<H<\frac{1}{2}$, equation \eqref{oms} yields
\begin{equation*}
	 J(\phi,\dot{\phi})=-\frac{1}{2}\int_0^1  (s^{-\alpha} I_{0^+}^\alpha s^\alpha \phi_s'-s^{-\alpha} I_{0^+}^\alpha s^\alpha b(\phi_s,\mathcal{L}_{X_s} ))^2+d_H b_x(\phi_s,\mathcal{L}_{X_s})ds.
\end{equation*}
	Suppose $\phi^*$ is an extremum. For any $\eta\in C^\infty_0([0,1])$, the first variation via the variational principle combined with \eqref{fffubini} gives
	\begin{align*}
  & -2\lim_{\varepsilon\to 0} \frac{J(\phi^*+\varepsilon \eta)-J(\phi^*)}{\varepsilon}\\ 
  = & \int_0^1 
  	2(s^{-\alpha} I_{0^+}^\alpha s^\alpha (\phi_s^*)'
  	-s^{-\alpha} I_{0^+}^\alpha s^\alpha b(\phi_s^*,\mathcal{L}_{X_s} ))(s^{-\alpha} I_{0^+}^\alpha s^\alpha \eta_s'-s^{-\alpha} I_{0^+}^\alpha s^\alpha b_x(\phi_s^*,\mathcal{L}_{X_s} )\eta_s)\\
  &\quad\quad\quad+d_H b_{xx}(\phi_s^*,\mathcal{L}_{X_s} )\eta_s ds\\
  =&  \int_0^1 2[s^\alpha I^{\alpha}_{1^-} s^{-2\alpha} I_{0^+}^\alpha s^\alpha ((\phi_s^*)'-b(\phi_s^*,\mathcal{L}_{X_s} ))](\eta_s'-b_x(\phi_s^*,\mathcal{L}_{X_s} )\eta_s)+d_H b_{xx}(\phi_s^*,\mathcal{L}_{X_s} )\eta_s ds\\
  =& \int_0^1 -2(\frac{d}{ds}+b_x(\phi_s^*,\mathcal{L}_{X_s} ))[s^\alpha I^{\alpha}_{1^-} s^{-2\alpha} I_{0^+}^\alpha s^\alpha ((\phi_s^*)'-b(\phi_s^*,\mathcal{L}_{X_s} ))]\eta_s+d_H b_{xx}(\phi_s^*,\mathcal{L}_{X_s} )\eta_s ds\\
  =&0.
\end{align*}
Therefore, the most probable path $\phi^*$ connecting states $x_0$ and $x_1$ must satisfy the Euler-Lagrange equation:
	\begin{equation}
		2\left(\frac{d}{ds}+b_x(\phi_s^*,\mathcal{L}_{X_s} )\right)\left[s^\alpha I^{\alpha}_{1^-} s^{-2\alpha} I_{0^+}^\alpha s^\alpha ((\phi_s^*)'-b(\phi_s^*,\mathcal{L}_{X_s} ))\right]-d_H b_{xx}(\phi_s^*,\mathcal{L}_{X_s} )=0,
	\end{equation}
	subject to the boundary conditions $\phi_0^*=x_0$, $\phi_1^*=x_1$. 
	
	For the case \( H > \frac{1}{2} \), the Euler-Lagrange equation follows from a similar method:
\begin{equation}
2\left( \frac{d}{ds} + b_x(\phi_s^*,\mathcal{L}_{X_s} ) \right) \left[ s^{-\alpha} D^{\alpha}_{1^-} s^{2\alpha} D_{0^+}^\alpha s^{-\alpha} \bigl( \phi_s'^* - b(\phi_s^*,\mathcal{L}_{X_s} ) \bigr) 
    \right] 
 - d_H b_{xx}(\phi_s^*,\mathcal{L}_{X_s} ) = 0,
\end{equation}
governed by the boundary conditions \( \phi_0^* = x_0 \), \( \phi_1^* = x_1 \).

	Below, we will demonstrate the application of the Onsager-Machlup functional in finding the most probable paths using specific one-dimensional and two-dimensional equations, along with numerical simulations.	
\begin{example}
	Consider the following stochastic differential equation 
\begin{equation}
    X_t =\pi+\int_0^t \sin(\mathbb{E}(X_s))ds + B^H_t, \quad t \in [0,1],\ H \in (1/4,1),
    \label{exam}
\end{equation}
where the drift term is defined as
\[
    b(x,\mu) = \sin\left(\int_{\mathbb{R}} y\ \mu(dy)\right).
\]
Under hypotheses \textnormal{(\textbf{Hyp 1})} and \textnormal{(\textbf{Hyp 2})}, this system exhibits distinct dynamical behaviors. In the noise-free scenario, the deterministic equation
\[
    dX_t = \sin(X_t)dt
\]
admits $x = \pi$ as a stable equilibrium point. However, under noise perturbation, transitions from $\pi$ to other states become possible. It is not difficult to observe that \( X_t = \pi + B_t^H \) is the solution of \eqref{exam}.
We investigate the most probable transition path from \( \pi \) to \( 2 \) governed by the following fractional differential equations:

For \( \frac{1}{2}\leq H   \):
\begin{equation}
    \frac{d}{ds}s^{-\alpha}D^\alpha_{1^-}s^{2\alpha}D^\alpha_{0^+}s^{-\alpha}\phi_s' 
    = 0;
    \label{Hhigh}
\end{equation}

For \( \frac{1}{4}<H < \frac{1}{2} \):
\begin{equation}
    \frac{d}{ds} s^{\alpha}I^\alpha_{1^-}s^{-2\alpha}I^\alpha_{0^+}s^{\alpha}\phi_s' 
       = 0,
    \label{Hlow}
\end{equation}
subject to boundary conditions $\phi(0) = \pi$, $\phi(1) = 2$. 

By solving these fractional ODEs \eqref{Hhigh} and \eqref{Hlow}, we obtain the most probable paths for different values of $H=3/10, 1/2, 7/10$. The comparison between the numerically simulated true paths and the most probable path are displayed in Figure \ref{fig1}. When \(H=1/2\), the most probable path is a linear function. When \( H = 3/10 \), the noise exhibits negative correlation, and the most probable path is concave downward near \( 0 \) and concave upward near \( 1 \). When \( H = 7/10 \), the noise exhibits positive correlation, and the most probable path has the opposite convexity.

\end{example}

	\begin{example}
		Consider the following 2-dimentional stochastic differential equation
		\begin{equation}
			\begin{bmatrix}
				X_t \\
				Y_t
			\end{bmatrix}
			=\begin{bmatrix}
				-\frac{\pi}{2}\\
				0
			\end{bmatrix} +\int_0^t\begin{bmatrix}
				Y_s\\
				-\frac{\Gamma(\frac{1}{4})^4}{4\pi} \sin(\mathbb{E}( X_s))
			\end{bmatrix}ds+							B^H_t
		 .\label{ex21}
		\end{equation}
		 When noise is absent, system \eqref{ex21} reduces to the classical undamped simple pendulum equation
		 \begin{equation}
Z'' + \frac{\Gamma(\frac{1}{4})^4}{4\pi} \sin Z= 0, \label{exam22}
\end{equation}
	whose phase portrait is shown in Fig \ref{fig2}.
		
	 It is straightforward to verify that equation \eqref{ex21} satisfies \textnormal{(\textbf{Hyp 1})} and \textnormal{(\textbf{Hyp 2})}. Applying Theorems \ref{nsincase} and \ref{nrecase}, we derive the Onsager-Machlup functional for \eqref{ex21} in the cases of $H\geq 1/2$ and $1/4<H<1/2$, respectively
	\begin{align}
		 J(\phi)&=-\frac{1}{2}\int_0^1 
		 \left( [s^{\alpha} D_{0^+}^\alpha s^{-\alpha} (\phi_{1,s}'-\phi_{2,s})]^2
		  +\left[s^{\alpha} D_{0^+}^\alpha s^{-\alpha}\left(\phi_{2,s}'+\frac{\Gamma(\frac{1}{4})^4}{4\pi} \sin(\mathbb{E}( X_s))\right)\right]^2
		  \right)ds, \quad \label{ex23}\\
		   J(\phi)&=-\frac{1}{2}\int_0^1 
		  \left([s^{-\alpha} I_{0^+}^\alpha s^\alpha (\phi_{1,s}'-\phi_{2,s})]^2
		  +\left[s^{-\alpha} I_{0^+}^\alpha s^\alpha\left(\phi_{2,s}'+\frac{\Gamma(\frac{1}{4})^4}{4\pi} \sin(\mathbb{E}( X_s))\right)\right]^2 
		  \right)ds.\label{ex24}
	\end{align}
	
	Through expectation analysis of \eqref{ex21}, we establish that 
$\mathbb{E}[X_s]$ satisfies the deterministic pendulum equation \eqref{exam22}. 
Energy conservation arguments reveal that solutions $Z_t$ of \eqref{exam22} 
with initial conditions $[Z(1), Z'(1)]^\mathsf{T} = [\pi/2, 0]^\mathsf{T}$ 
maintain constant total energy.
	It is easy to deduce from \eqref{ex23} and \eqref{ex24} that for \( H \in (1/4,1) \), the most probable transition paths of \eqref{ex21} connecting 
$[-\pi/2, 0]^\mathsf{T}$ to $[\pi/2, 0]^\mathsf{T}$ correspond precisely to 
solutions of the deterministic system \eqref{exam22}.
The comparison between the numerically simulated true paths and the most probable paths can be seen in Fig.\ref{fig3}.

\end{example} 

	\begin{figure*}[h]
		\centering
		\includegraphics[width=0.7\textwidth]{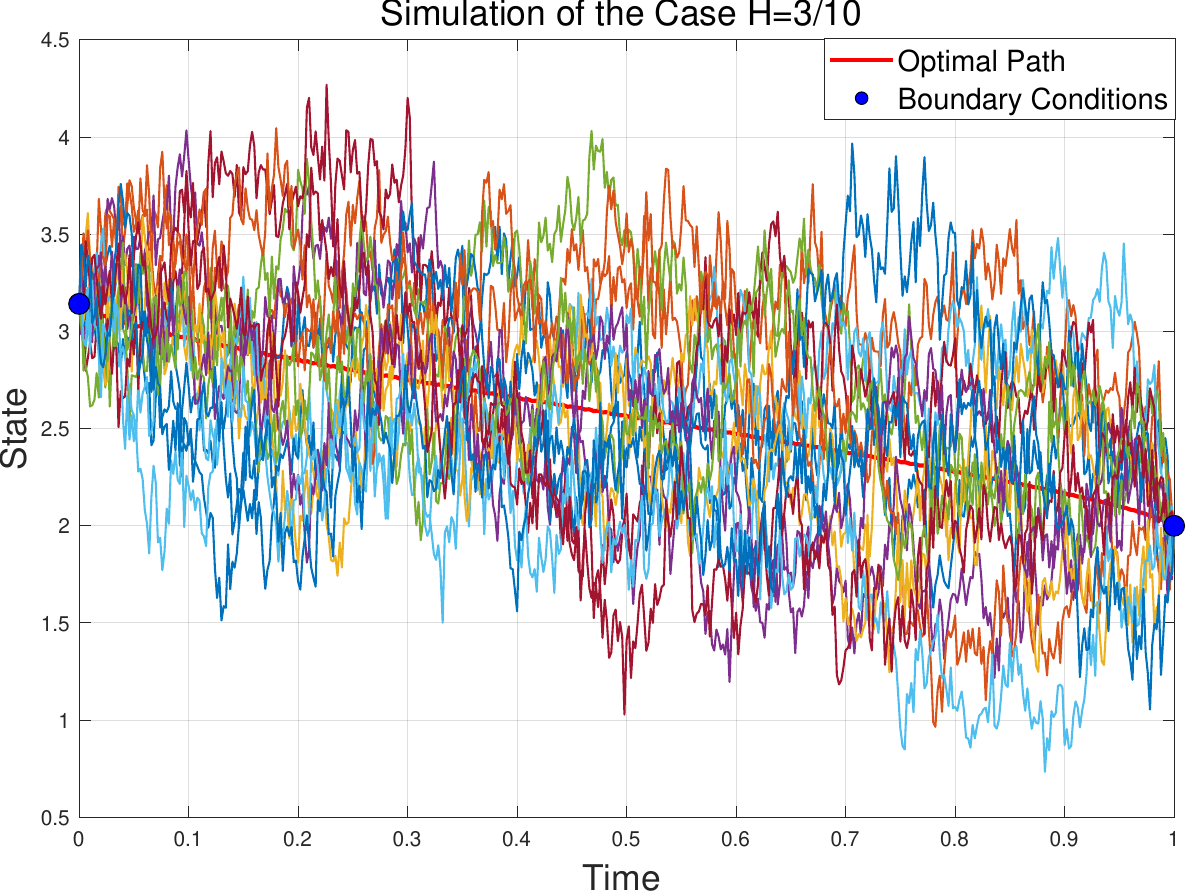}
	\end{figure*}

	\begin{figure*}[h]
		\centering
		\includegraphics[width=0.7\textwidth]{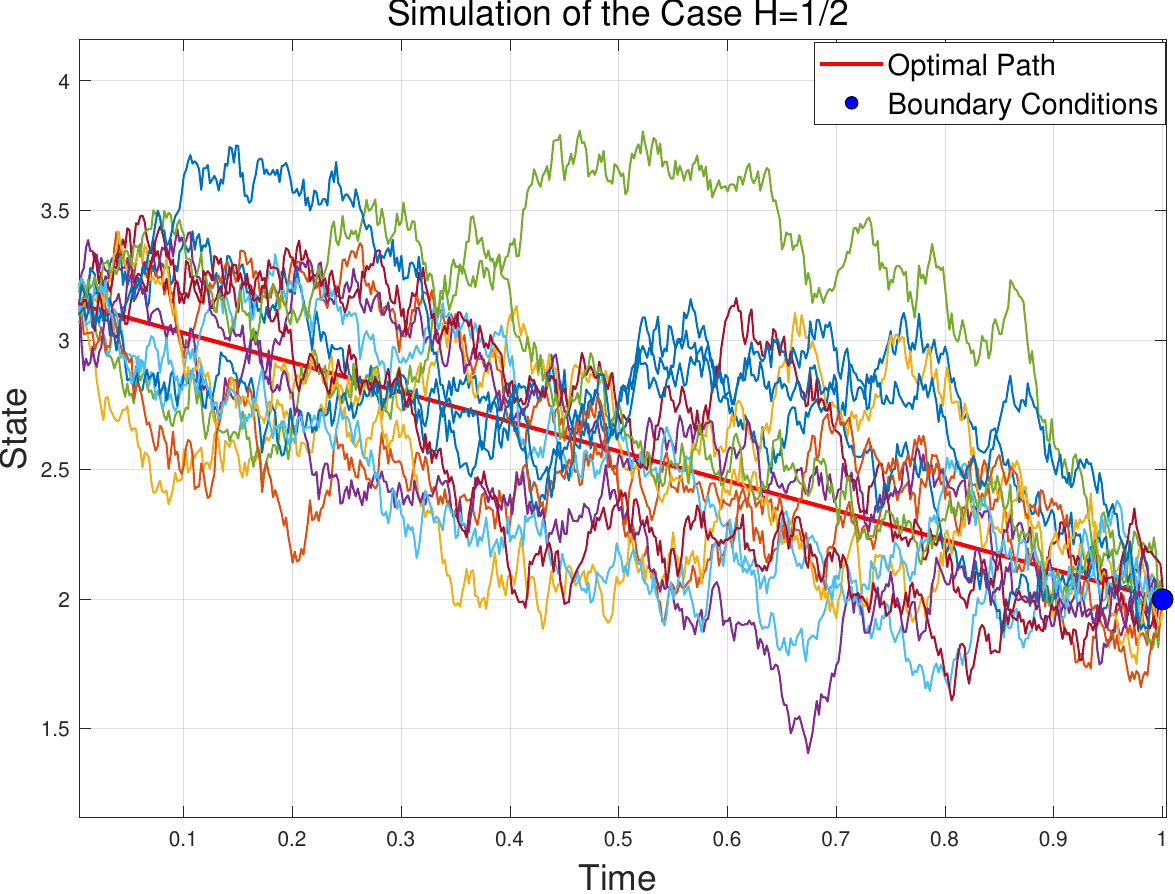}
	\end{figure*}
	\begin{figure}[h]
		\centering
		\includegraphics[width=0.7\textwidth]{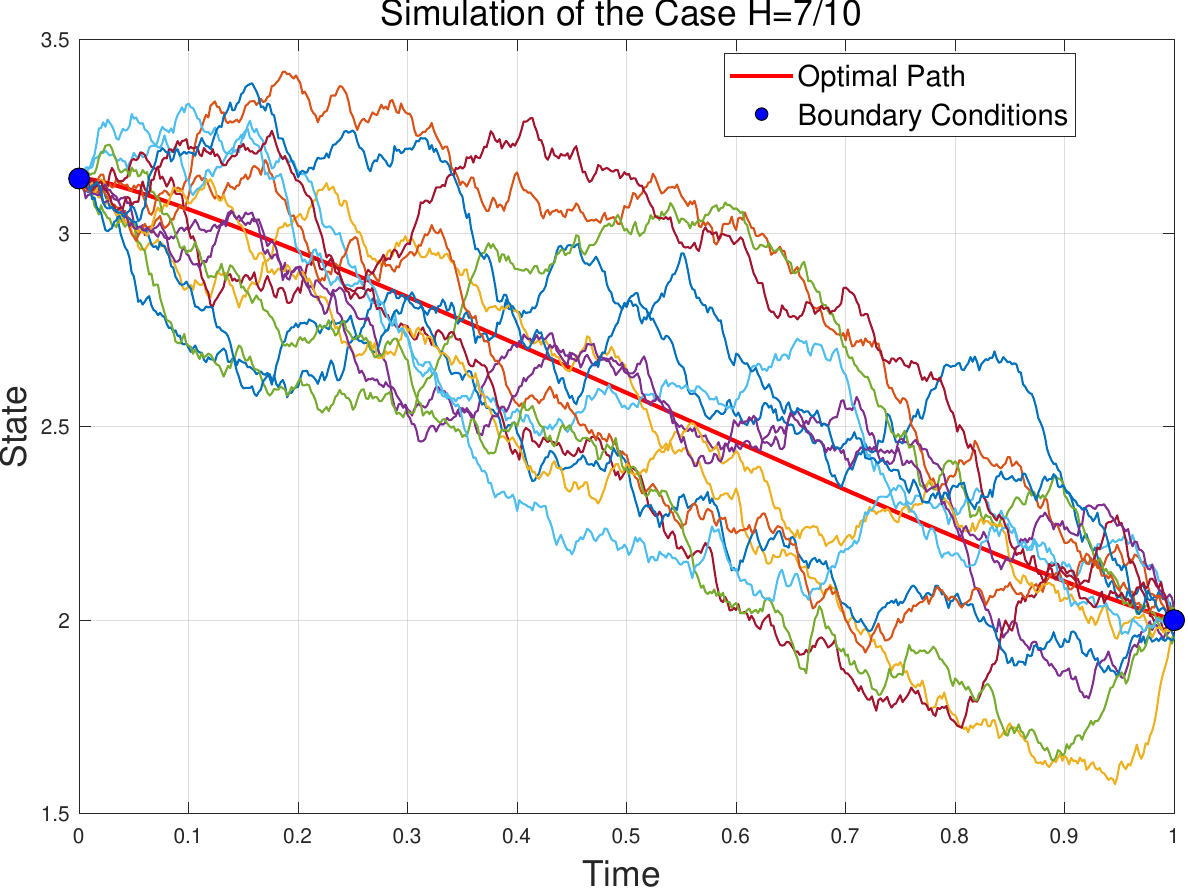}
		\caption{The stochastic process \( X_t \) follows equation \eqref{exam}, transitioning from the initial state \( x = \pi \) to the final state \( x = 2 \). The most probable path, \( \phi_t \) (depicted by the red line), is illustrated for \( H = \frac{3}{10} \), \( H = \frac{1}{2} \), and \( H = \frac{7}{10} \).}
    \label{fig1}
	\end{figure}

\begin{figure}[h]  
    \centering
    \includegraphics[width=0.7\textwidth]{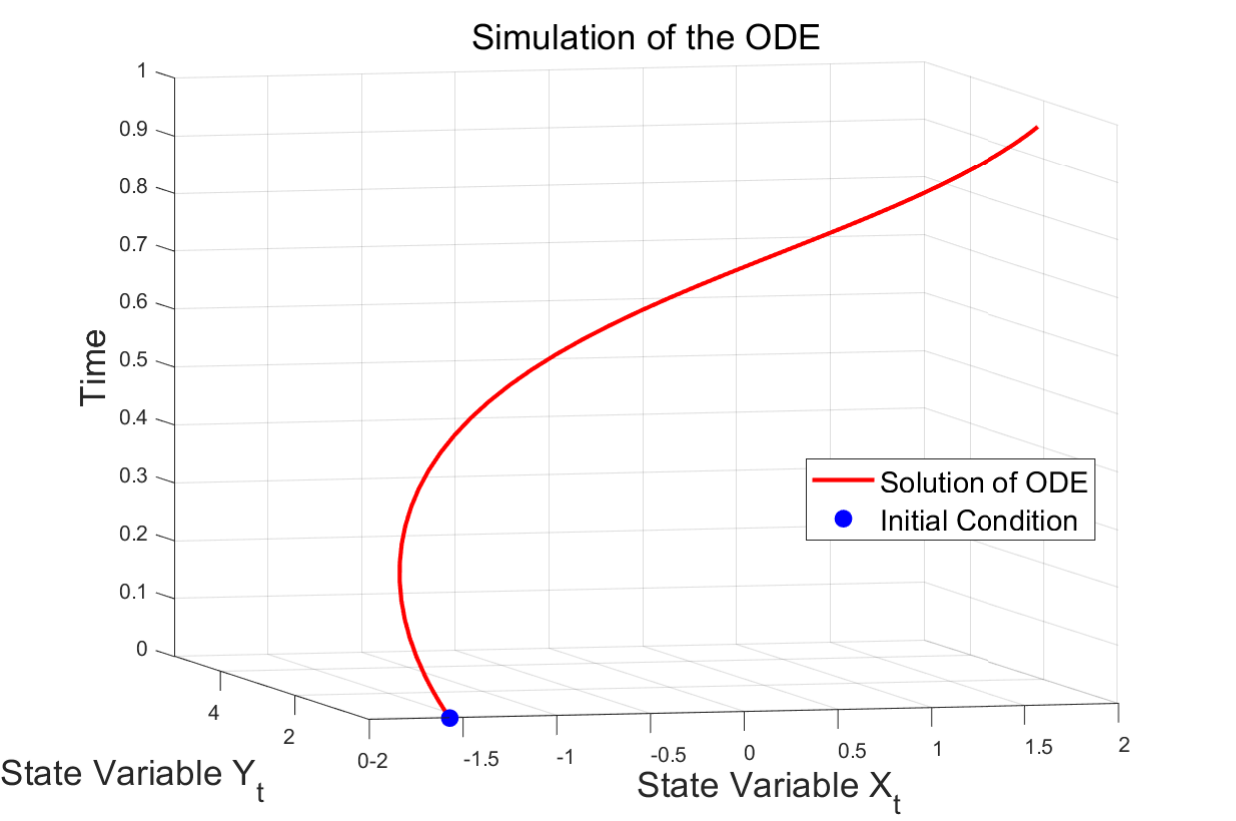}
    \caption{Phase portrait of the deterministic pendulum system \eqref{exam22}}
    \label{fig2}
\end{figure}

	\begin{figure*}[h]
		\centering
		\includegraphics[width=0.7\textwidth]{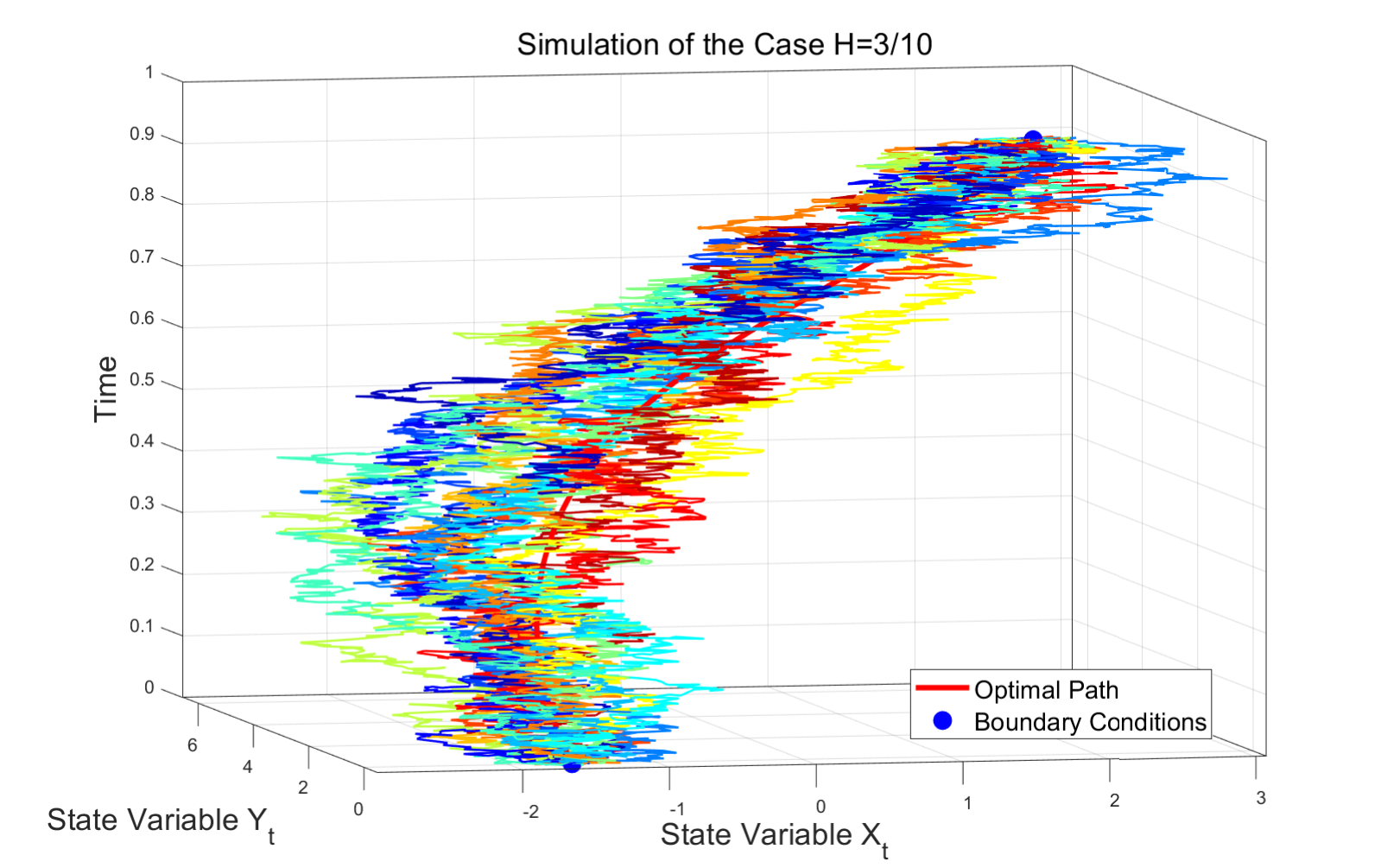}
	\end{figure*}

	\begin{figure*}[h]
		\centering
		\includegraphics[width=0.7\textwidth]{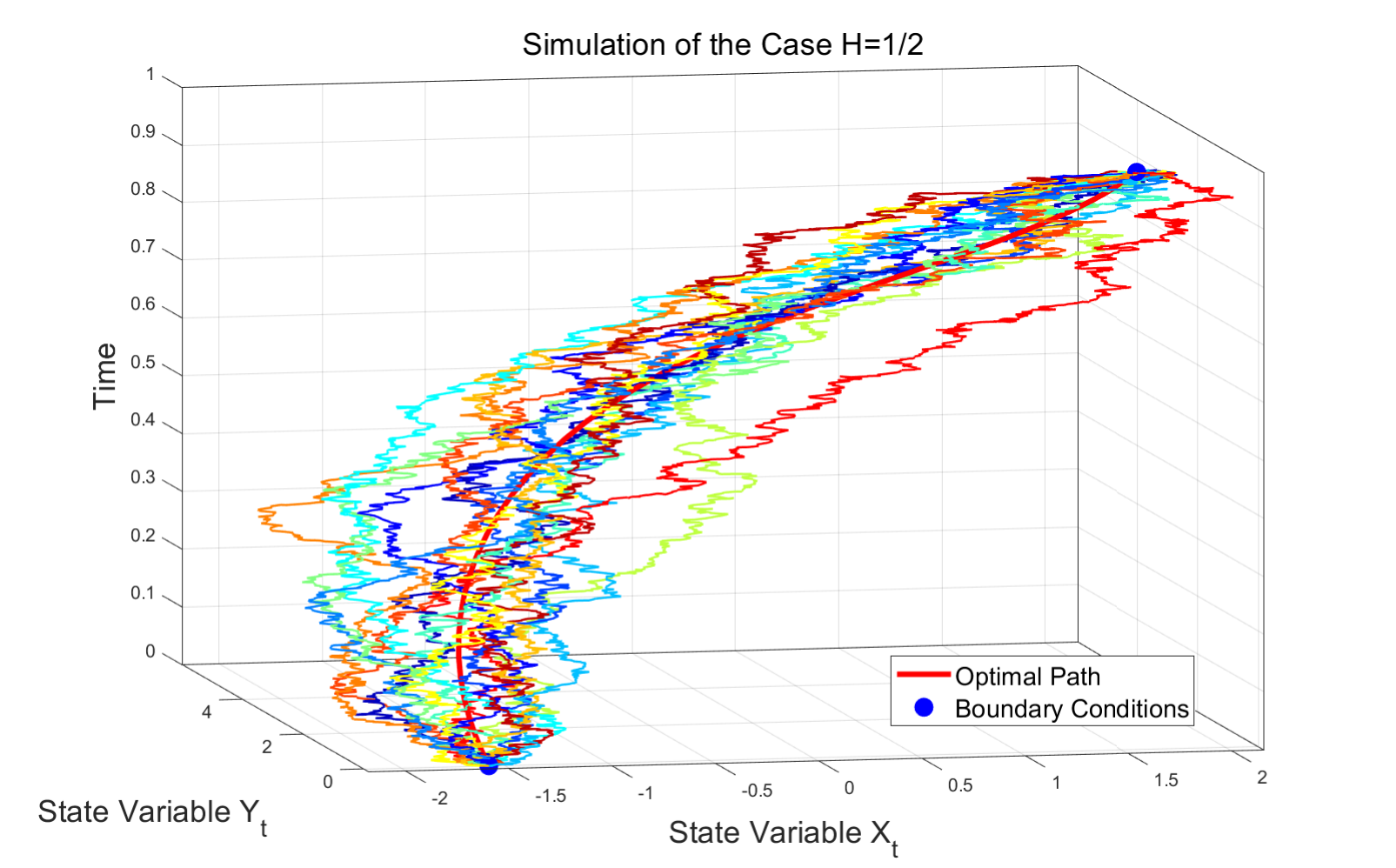}
	\end{figure*}
	
	\begin{figure}[h]
		\centering
		\includegraphics[width=0.7\textwidth]{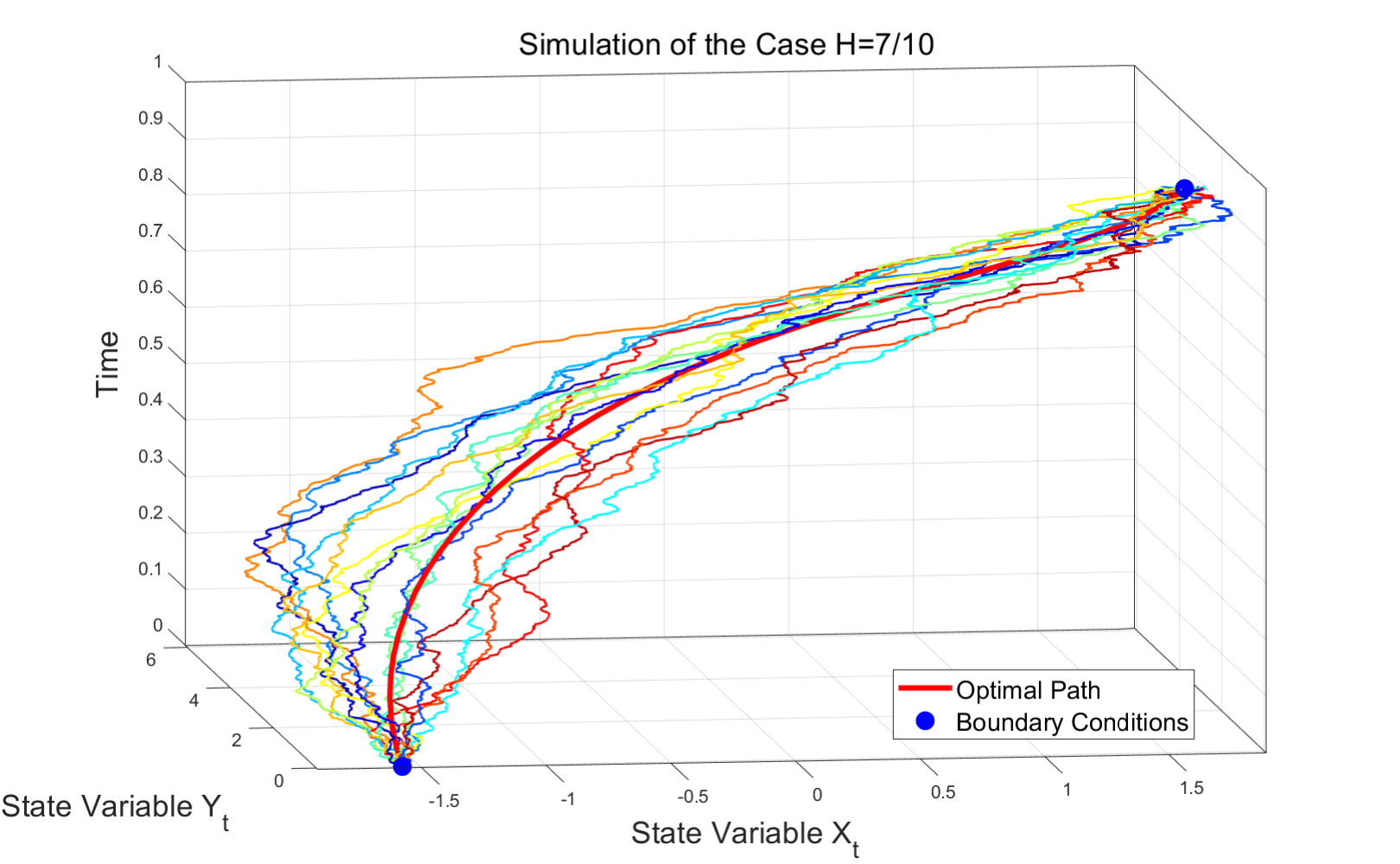}
		\caption{The stochastic process \( X_t \) follows equation \eqref{ex21}, transitioning from the initial state $(-\frac{\pi}{2},0)$ to the final state $(\frac{\pi}{2},0)$. The most probable path, \( \phi_t \) (depicted by the red line), is illustrated for \( H = \frac{3}{10} \), \( H = \frac{1}{2} \), and \( H = \frac{7}{10} \).}
    \label{fig3}	
    \end{figure}

\FloatBarrier
\bibliographystyle{plain}
\bibliography{math}

\end{document}